\documentclass[11pt,reqno]{amsart}
\usepackage{amssymb,amsthm,amsmath,amstext,amsxtra}
\usepackage{fullpage}
\usepackage{pifont}
\usepackage[all]{xy}
\usepackage{booktabs}
\usepackage{mathtools}
\usepackage{stackengine}
\usepackage{hyperref}
\hypersetup{colorlinks=true,urlcolor=blue,citecolor=blue,linkcolor=blue}
\usepackage{enumerate}
\usepackage[shortlabels]{enumitem}
\usepackage{comment}
\usepackage{multirow}
\usepackage{colonequals}
\usepackage{array}
\usepackage{tikz}
\usepackage{marginnote}
\usepackage[export]{adjustbox}

\numberwithin{equation}{subsection}

\theoremstyle{plain}
\newtheorem*{theorem*}{Theorem}
\newtheorem{theorem}[equation]{Theorem}

\newtheorem{proposition}[equation]{Proposition}

\newtheorem{lemma}[equation]{Lemma}

\newtheorem*{corollary*}{Corollary}
\newtheorem*{conjecture*}{Conjecture}
\newtheorem{conjecture}[equation]{Conjecture}

\theoremstyle{definition}
\newtheorem{definition}[equation]{Definition}
\newtheorem*{definition*}{Definition}

\theoremstyle{remark}
\newtheorem{remark}[equation]{Remark}

\newenvironment{enumalph}
{\begin{enumerate}}
{\end{enumerate}}

\newcommand{\Z}{\mathbb{Z}}

\newcommand{\Q}{\mathbb{Q}}

\newcommand{\R}{\mathbb{R}}

\newcommand{\F}{\mathbb{F}}

\newcommand{\cH}{\mathbf{H}}
\newcommand{\I}{\mathbf I}

\newcommand{\Un}{F_{un}}

\newcommand{\calO}{\mathcal{O}}
\newcommand{\AlF}{F_{p}}
\newcommand{\AlK}{F_{p}}
\newcommand{\cO}{\calO}

\newcommand{\defi}[1]{\textsf{#1}} 	

\DeclareMathOperator{\Cl}{Cl}

\DeclareMathOperator{\primes}{primes}

\DeclareMathOperator{\Prob}{Prob}

\DeclareMathOperator{\Norm}{Norm}

\DeclareMathOperator{\Avg}{Avg}

\DeclareMathOperator{\Aut}{Aut}

\DeclareMathOperator{\Inertia}{Inertia}

\newcommand{\vect}{w} 
\newcommand{\FK}{F} 
\newcommand{\Spl}{S}

\newcommand{\can}{\mathrm{can}}

\DeclareMathOperator{\Disc}{\mathrm{Disc}}
\DeclareMathOperator{\rk}{\mathrm{rk}}

\DeclareMathOperator{\Sel}{\mathrm{Sel}}
\DeclareMathOperator{\sgn}{\mathrm{sgn}}
\DeclareMathOperator{\even}{\mathrm{even}}

\newcommand{\fraka}{\mathfrak{a}}
\newcommand{\fa}{\fraka}

\newcommand{\xmark}{\ding{55}}

\usepackage{tikz}
\usepackage{tikz-cd}
\usepackage{makecell}
\usepackage{xcolor}
\definecolor{myblue}{RGB}{25,25,200}

\usepackage{hyperref}
\hypersetup{pdftitle={On unit signatures and narrow class groups of odd abelian number fields: Galois structure and heuristics},
pdfauthor={Benjamin Breen, Ila Varma, and John Voight}}
\hypersetup{colorlinks=true,urlcolor=blue,citecolor=blue,linkcolor=blue}

\begin{document}

\title[Heuristics for odd abelian fields]{On unit signatures and narrow class groups of odd abelian number fields: Galois structure and heuristics}

\author{Benjamin Breen}
\address{Department of Mathematics, Clemson University, 206 Long Hall, Clemson, SC 29631}
\email{bkbreen@clemson.edu}
\urladdr{}

\begin{abstract}
This paper is an extension of the work of Dummit and Voight on modeling the 2-Selmer group of number fields. We extend their model to $S_n$-number fields of even degree and develop heuristics on the difference in the 2-ranks between the class group and narrow class group.  
\end{abstract}

\title{The 2-Selmer Group of $S_n$-number fields of even degree}

\maketitle 
\tableofcontents

\section{Introduction}

\subsection{Motivation}

Their are several extensions of the Cohen-Lenstra-Martinet heuristics to the distributions of narrow class groups. However, these extensions are mainly for odd degree number fields, in particular abelian \cite{BVV} and $S_n$-number fields \cite{DV}. Predictions for narrow class groups of number fields of even degree are restricted to the case of real quadratic fields \cite{FouvryKluners, Stevenhagen}. 

In this paper, we provide heuristics on the relationship between the 2-torsion subgroup of the class group and the narrow class group for $S_n$-number fields of even degree. These heuristics stem from the 2-Selmer group of a number field $\FK$ \begin{equation*}\label{eqdef:2Sel} \Sel_2(\FK) \colonequals \{ \alpha \in \FK^{\times} : (\alpha) = \mathfrak{a}^2 \text{ for a fractional ideal } \mathfrak{a} \mbox{ of }\FK\} /(\FK^{\times})^2 \end{equation*} 
and a study of the image of $\Sel_2(F) \leq F^\times/F^{\times2}$ under all of the various global-to-local maps $F^\times/ F^{\times2} \to F_v^\times / F_v^{\times2}$ where $F_v$ is the completion of $F$ at an Archimedean or 2-adic place $v$. This can be neatly packaged into the 2-Selmer signature homomorphism 
\begin{equation*} \sgn:  \Sel_2(\FK)  \to V_\infty(\FK) \boxplus V_2(\FK),   \end{equation*} 
where $V_\infty(\FK)$ and $V_2(\FK)$ are local signature spaces constructed from the Archimedean and 2-adic algebras of $F$. We develop a model for the image $S \colonequals \sgn \Sel_2 (\FK)$ as $F$ varies over $S_n$-number fields of even degree, from which we derive precise predictions on the difference between the 2-torsion subgroups of the class group and narrow class group.

\subsection{Results}

 The difficulty in developing heuristics on 2-torsion for number fields of even degree is \emph{genus theory}: the contribution to the class field theory of a number field coming from the class field theory of its proper subfields. For an extension of number fields $\FK/E$, let the $E$-\defi{imprimitive part} of the 2-Selmer group of $\FK$ be \begin{equation*} \Sel_2^{E}(\FK) \colonequals \Sel_2(\FK) \cap E^\times \FK^{\times2} / (\FK^{\times})^2, \end{equation*} consisting of those classes with a representative $\alpha \in E^\times$. 
 

 We say that a number field is \defi{primitive} if it has no nontrivial subfields. The effects of genus theory for a primitive number field are restricted to the $\Q$-imprimitive part of the 2-Selmer group. We now provide a full classification for the image of the $\Q$-imprimitive part of the Selmer group  $S^\Q \colonequals \sgn \Sel_2^{\Q}(\FK)$. 
 
\begin{theorem}[Theorem \ref{thm::Q-types}] \label{thm::into1} Let $F$ be a number field of degree $n$.  
\begin{enumalph} \item When $n$ is odd, there is $1$ possibility for $S^\Q$. \item When $n$ is even, there are $6$ possibilities for $S^\Q$ up to isometry. \end{enumalph} \end{theorem}

We label the 6 possibilities in Theorem \ref{thm::into1}(b) as $\Q$-imprimitive types A(i)-(ii) and B(i)-(iv). For all of these types to occur, it is necessary, but not sufficient that the fields has at least one real place and so we will mainly focus on fields that are not totally complex. 

We investigated the proportion of $S_n$-number fields with each $\Q$-imprimitive type.  When real quadratic fields are ordered by discriminant, then 1/6 of have $\Q$-imprimitive type A(ii) and 5/6 have $\Q$-imprimitive type B(iii).  We prove the following result for $S_4$-number fields. 

\begin{theorem} \label{thm::S4fields}  As $F$ varies over totally real $S_4$-number fields ordered by absolute discriminant, then 
\textnormal{\setlength{\tabcolsep}{10pt}
\def\arraystretch{1.5}
\begin{center}
\begin{tabular}{ c || c | c| c| c| c| c  } 
$\Q$-imprimitive types &  A(i) &  A(ii) &  B(i) & B(ii) & B(iii)  & B(iv) \\ \hline \hline 
Proportion of fields &  0.0018 & 0.0423 & 0.7280 & 0.0996 & 0.1138 & 0.0143  \\ 
\end{tabular}
\end{center}}
\end{theorem}

When $n > 4$,  we provide a set of explicit predictions (Conjecture \ref{conj:fullclass}) on the proportion of fields with each $\Q$-imprimitive type based on Bhargava's conjectures on local algebras \cite{Bhar}.  These prediction indicate that as the degree grows large most $S_n$-number fields will have $\Q$-imprimitive type B(i). This $\Q$-imprimitive is the only one where which the class group and the narrow class group could be equal, hence a positive proportion (at least $27.19\%$) of totally real or mixed signature $S_4$-number fields have a narrow class group that is strictly larger than the class group.

\subsection{Conjectures} 

Let $\FK$ be a number field with class group $\Cl(\FK)$ and narrow class group $\Cl^+(\FK)$. For an abelian group $A$, let $A[2] \colonequals \{ a \in A : a^2 = 1\}$ be the 2-torsion subgroup of $A$. Dummit--Voight \cite{DV} studied the \defi{isotropy rank} \[ k(\FK) \colonequals \dim_{\F_2} \Cl^+(\FK)[2] - \dim_{\F_2}\Cl(\FK)[2] \] 
and developed heuristics for isotropy ranks of $S_n$-number fields of odd degree which stemmed from the assumption that the image $S \subseteq V_2(\FK) \boxplus V_\infty(\FK)$ behaved like a random maximal totally isotropic subspace. This was later extended to abelian number fields of odd degree by Breen--Varma--Voight \cite{BVV} who incorporated a Galois module structure into their model. 

We extend this model to $S_n$-number fields of even degree. To do this, we treat the image of the 2-Selmer group as random maximal totally isotropic subspace containing a subspace $S^\Q \subseteq S$ determined by the $\Q$-imprimitive type. This is done in generality in Theorem \ref{thm::isoprobs}; we now state the specific predictions for $\Q$-imprimitive type B(i). 
 Let
\begin{equation*} (q)_m \colonequals \prod_{i=1}^m (1-q^{-i}),\hspace{1cm} (q)_0 \colonequals 1 \end{equation*} be the $q$-Pochhammer symbol.

\begin{conjecture}[Conjecture \ref{conj::isotropyranks}] \label{thm::isoprobsinto} 
For $k \in \Z_{\geq0}$, let 
\begin{eqnarray*} C (k) & \colonequals & \frac{1}{2^{k(k+r_2)}} \cdot  \frac{ (2)_{r_1+r_2 -1}  (4)_{r_1/2-1} (4)_{r_1/2 - 1 + r_2 } }{  (2)_{k} (2)_{k + r_2}(4)_{r_1+r_2 -1} (4)_{r_1/2-1-k}} \end{eqnarray*}
As $F$ varies over $S_n$-number fields of even degree with signature $(r_1,r_2)$ and $\Q$-imprimitive type B(i) ordered by absolute discriminant, then 
\begin{eqnarray*} \Prob( F \: | \: k(F) = k ) & = &  \begin{cases} C(k) & k = 0; \\
C(k) + C(k-1)/2^{r_1+r_2-1}  & 0<k < r_1/2; \\
C(k-1)/2^{r_1+r_2-1}  & k = r_1/2. \end{cases} \end{eqnarray*}
\end{conjecture}

Lastly, we consider moments. Malle \cite{Malle,Malle1} provides heuristics on the 2-part of the class group for $S_n$-number fields of odd degree, predicting that the average number of 2-torsion elements in $\Cl(\FK)$ is $1+2^{1-r_1-r_2}$.  Using Malle's predictions, Dummit--Voight \cite{DV} conjectured that the average number of 2-torsion elements in $\Cl^+(\FK)$ is $1+2^{-r_2}$.

 Currently, there are no predictions for the 2-part of the class group for $S_n$-number fields of even degree.  We investigated $S_n$-number fields for which the contribution from genus theory is as minimal as possible; namely $\Sel_2^\Q(\FK) = \{\pm1\}$. We say that these fields have \defi{trivial $\Q$-imprimitive type}. We propose the following conjecture for these fields.

\begin{conjecture} \label{cor::genusunobstructedtwotorsioneltsintro} Let $\FK$ vary over $S_n$-number fields with signature $(r_1,r_2)$ and trivial $\Q$-imprimitive type ordered by absolute discriminant. \begin{enumalph} \item The average number of 2-torsion elements in $\Cl(\FK)$ is $1+ 2^{1-r_1 - r_2}$.\item The average number of 2-torsion elements in $\Cl^+(\FK)$ is $ 1+ 2^{-r_2}$ when $r_1 > 0$. \end{enumalph} \end{conjecture} 

While every $S_n$-number fields of odd degree has trivial $\Q$-imprimitive type, it is clear that no quadratic fields have trivial $\Q$-imprimitive type.  We investigated the proportion  of $S_n$-number fields with even degree with trivial $\Q$-imprimitive type.  Let \begin{equation*} c(n,p)  =  \sum_{k =0}^{n-1} q(k,n-k)p^{-k}  \end{equation*}  where $q(k,n-k)$ denotes the number of partitions of $k$ into at most $n-k$ parts.

\begin{conjecture} \label{conj::Qtriv}  As $\FK$ varies over $S_n$-number fields of degree $n = 2m$ ordered by absolute discriminant, the density of fields with trivial $\Q$-imprimitive type (independent of signature) is 
\begin{eqnarray*}  \prod_{\primes \: p}   \left (1 - \frac{c(m,p)}{c(2m,p)}  \cdot \frac{1}{p^{m}} \right ). \end{eqnarray*} 
    \end{conjecture} 
See \S\ref{subsec::trivialtypes} for a table of values of Conjecture \ref{conj::Qtriv} when $4\leq n \leq 20$.  In particular, Conjecture \ref{conj::Qtriv} predicts that as the degree $n$ grows large then approximately $100\%$ of $S_n$-number fields will have trivial $\Q$-imprimitive type. 

\subsection{Layout}

This paper is organized as follows. In $\S$\ref{Section::2}, we provide a short background on the 2-Selmer group of a number field and its relation to ray class fields. In $\S$\ref{Section::3}, we prove a classification of the $\Q$-imprimitive types of $S_n$-number fields of even degree. We then expand upon this in $\S$\ref{Section::4} by working out the conjectural density for each $\Q$-imprimitive types based on Bhargava's conjectures for local algebras of $S_n$-number fields \cite{Bhar}. In $\S$\ref{Section::5} we provide a count for the number of maximal totally isotropic subspaces with a fixed $\Q$-imprimitive subspace $S^\Q$. Lastly, in $\S$\ref{Section::6} we provide computational evidence for $S_4$-number fields to support our conjectures. 

\subsection{Acknowledgements}

The author was supported by the Simons Collaboration Grant (0550029 to John Voight) and RTG Grant DMS $\#1547399$. We thank John Voight for his comments and feedback in the preparation of this manuscript.

 \section{Background and notation} \label{Section::2} 

\subsection{Notation}

Let $\FK$ be a number field of degree $n \colonequals [\FK:\Q]$ with $r_1$ real places and $r_2$ complex places. For $v$ a place of $\FK$, and $\alpha \in \FK$, we abbreviate $\alpha_v \colonequals v(\alpha)$.  We say that $\FK$ is an \defi{$S_n$-number field} if its Galois closure has Galois group isomorphic to $S_n$. 
  
\subsection{The 2-Selmer group and its signature spaces} \label{sec:2selsigdefs}

\begin{definition} The \defi{$2$-Selmer group} of a number field $\FK$ is \begin{equation} \Sel_2(\FK) \colonequals \{ \alpha \in \FK^{\times} : (\alpha) = \fraka^2 \text{ for a fractional ideal } \fraka \mbox{ of }\FK\} /(\FK^{\times})^2. \end{equation} \end{definition}

Following Dummit and Voight  \cite{DV}, we recall two signature spaces that record the signs of elements in $\FK$ at the archimedean and 2-adic places. 

\begin{definition}\label{def:VooK} The \defi{archimedean signature space} $V_\infty(\FK)$ of $\FK$ is \begin{equation*} V_\infty(\FK) \colonequals \prod_{\text{$v$ real}} \FK_v^\times/ \FK_v^{\times2}  \simeq  \prod_{\text{$v$ real}} \{\pm 1\} \end{equation*} and the \defi{archimedean signature map} $\sgn_\infty \colon \FK^\times/(\FK^{\times})^2 \to V_\infty(\FK)$ sends a representative $\alpha \in \FK^\times$ to $\sgn(\alpha)$, the tuple recording the sign $\alpha_v/ |\alpha_v|$ at each real place $v$. We treat $V_\infty(\FK)$ as an $\F_2$-vector space of dimension $r_1$;  in the case that $r_1 = 0$ then $V_\infty(F) \colonequals \{0\}$ is the trivial $\F_2$-vector space. The product of Hilbert symbols over the real places defines a map \begin{equation} \begin{aligned} b_\infty \colon V_\infty(\FK) \times V_\infty(\FK) & \to \{\pm 1\}, \end{aligned} \end{equation} which is a nondegenerate, symmetric, bilinear form on $V_\infty(\FK)$ so long as $r_1>0$.   \end{definition}

Let $\alpha \FK^{\times2}$ be a representative for a square class in $\Sel_2(\FK)$. For a nonarchimedean place $v$, then $\alpha_v = \pi_v^{2n} u_v$ where $\pi_v$ is a choice of local uniformizer, $u_v \in \cO_v^\times$ is a unit in the local ring of integers $\cO_v$, and $n \in \Z$. Therefore, $\alpha$ is locally represented by the square class $u_v\FK_v^{\times2}$ in $\FK_v^\times/\FK_v^{\times2}$. 

\begin{definition}\label{def:V2} The \defi{2-adic signature space} $V_2(\FK)$ of $\FK$ is \begin{equation*} V_2(\FK) \colonequals \prod_{v\mid2}  \cO_v^\times / (1 + 4 \cO_{v}) \cO_{v}^{\times2} \end{equation*} and the \defi{$2$-adic signature map} is $\sgn_2 \colon \Sel_2(\FK) \to V_2(\FK)$ sends a representative $\alpha \FK^{\times2} \in \Sel_2(\FK)$ to the tuple of local units $(u_v)_{v \mid 2}$. The space $V_2(\FK)$ has the structure of an $\F_2$-vector space of dimension $n$. The product of Hilbert symbols over places above $2$ defines a map \begin{equation} \begin{aligned} b_2 \colon V_2(\FK) \times V_2(\FK) & \to \{\pm 1\},  \end{aligned}  \end{equation} which is a symmetric, nondegenerate, bilinear form on $V_2(\FK)$. \end{definition}
 
\begin{definition} The \defi{2-Selmer signature space} of $\FK$ is the orthogonal direct sum \begin{equation*} V(\FK)  \colonequals V_\infty(\FK) \boxplus V_2(\FK), \end{equation*} 
which comes with the \defi{$2$-Selmer signature map} $\sgn \colonequals (\sgn_{\infty}, \sgn_2) : \Sel_2(\FK) \to V(\FK)$ that fits into the  following diagram 
\[ \begin{tikzcd} & \Sel_2(\FK) \arrow[d, swap, "\sgn" ]   \arrow[dr, "\sgn_{2}"]  \arrow[dl, swap, "\sgn_{\infty}" ]  & \\
V_\infty(\FK)  \arrow[r]  & V_\infty(\FK) \boxplus V_2(\FK)   & V_2(\FK).   \arrow[l] 
\end{tikzcd} \]  \end{definition} 

The image of 2-Selmer group inside the 2-Selmer signature space $S(\FK) \colonequals\sgn(\Sel_2(\FK))$ plays a central role in this paper.  Dummit-Voight proved the following classification result for $S(\FK)$. 
\begin{theorem}\cite[Proposition 6.1.]{DV} \label{Thm:MaximalIsotropic} The image of the Selmer group $S(\FK) \subseteq V(\FK)$ is a maximal totally isotropic subspace. \end{theorem}

\subsection{Bilinear forms}

Let \begin{itemize} \item $\I \simeq \F_2$ with the \defi{dot product} bilinear form represented by the matrix $\begin{pmatrix}1 \end{pmatrix};$ \item $\cH \simeq \F_2^2$ with the \defi{hyperbolic plane} bilinear form represented by the matrix $\begin{pmatrix} 0 & 1 \\ 1 & 0 \end{pmatrix}.$ \end{itemize} 
Every finite dimensional vector space $V$ over $\F_2$ with a symmetric, nondegenerate, bilinear form $b$ is isometric to an orthogonal direct sum of copies of $\I$ and $\cH$. In order to classify these bilinear spaces up to isometry, we need to consider the canonical vector. 

Let $V^*$ be the dual space of $V$. Then $V^*$ contains a \defi{canonical linear functional} $\phi_{\can}  \colon V \to \F_2$ defined by $\phi_{\can}(w) = b(w,w)$ for $w \in V$. Since the bilinear form $b$ is nondegenerate, we have a canonical isomorphism $V \simeq V^*$ given by $w \mapsto b(w,\_)$. The \defi{canonical vector} $\vect_{\can} \in V$ is the unique vector that maps to $\phi_{\can}$ under the isomorphism $V \simeq V^*$. Equivalently, it is the unique vector satisfying $b(\vect_{\can},\vect) = b(\vect,\vect)$ for every vector $\vect \in V$.  Up to isometry, the bilinear space $V$ is determined by its dimension and the canonical vector.

\begin{lemma} \label{lem:bilspace} Let $V$ be a symmetric, nondegenerate, bilinear space over $\F_2$ of dimension $n$. 
 \begin{enumerate} \item  If $n = 2m +1$ is odd, then $V \simeq \cH^{m} \boxplus \I$ is nonalternating and $\vect_{\can}$ is anisotropic; \item If $n = 2m$ is even, then either:
 \begin{enumerate}  \item $V \simeq \cH^{m}$ is alternating and $\vect_{\can} =0$, or 
  \item $V \simeq \cH^{m-1} \boxplus \I^2$ is nonalternating and $\vect_{\can}$ is a nonzero isotropic vector.
\end{enumerate} \end{enumerate} \end{lemma}

Let $(\_,\_)_v$ denote the Hilbert symbol at a place $v$ of $\FK$. For $\alpha \in \FK_v^\times$, we have $(-1, \alpha)_v =  (\alpha, \alpha)_v$ so $\sgn_\infty(-1)$ is the canonical vector of $V_\infty(\FK)$ and $\sgn_2(-1)$ is the canonical vector of $V_2(\FK)$. Combining this with previous lemma yields the following result.

\begin{theorem}\cite[Proposition 5.2.]{DV} Let $b_2$ and $b_\infty$ be as given in definition \ref{def:VooK} and \ref{def:V2}. 
\begin{enumerate}
\item $b_\infty$ is nonalternating when $r_1 > 0$.
\item $b_2$ is alternating if and only if the degree $n$ is even and $\sgn_2(-1) = 0$.
\end{enumerate}
\end{theorem}

\subsection{Relation to class groups} 

For an abelian group $A$, let $\rk_2 A = \dim_{\F_2} A[2]$. There are four ray class groups and their corresponding class fields which are central to our investigation:
\begin{itemize}
\item $\Cl(\FK), H$ — the class group and Hilbert class field;
\item $\Cl^+(\FK), H^+$ — the narrow class group and narrow Hilbert class field;
\item $\Cl_4(\FK), H_4$ — the ray class group and class field of modulus $(4)$;
\item $\Cl^+_4(\FK), H^+_4$ — the narrow ray class group and class field of modulus $(4)\infty$.
\end{itemize}
where $\infty$ denotes the set of all real places.  The 2-Selmer group of a number field is related to the narrow ray class group of modulus (4) by 
\[ \rk_2 \Cl_4(\FK) = \dim_{\F_2} \Sel_2(\FK). \] Further, we have the following refinement of this relation in terms of the kernels of the  various signature maps as follows: 
\begin{eqnarray*}
\rk_2 \Cl(\FK) & = & \dim_{\F_2} \ker(\sgn), \\
\rk_2 \Cl^+(\FK) & = &  \dim_{\F_2} \ker(\sgn_2), \\
\rk_2 \Cl_4(\FK) & = & \dim_{\F_2} \ker(\sgn_\infty). 
\end{eqnarray*}

\section{Imprimitive part of the $2$-Selmer group} \label{Section::3}

\subsection{Imprimitive types} 

For an extension of number fields $\FK/E$, the $E$-\defi{imprimitive} part of the 2-Selmer group of $\FK$ is \[ \Sel_2^E(\FK) \coloneqq  \Sel_2(\FK) \cap  E^\times \FK^{\times2} / \FK^{\times2}.  \] Since the only proper subfield of an $S_n$-number fields is $\Q$, this section focuses on a classification of the $\Q$-imprimitive part of the 2-Selmer group for $S_n$-number fields of even degree.

\begin{proposition} \label{prop:reps} The $\Q$-imprimitive part of the $2$-Selmer group of a number field $\FK$ is generated by the following square classes \[ \Sel_2^\Q(\FK) =  \{ -1 \} \cup \{ p \in \Z  : \text{$p$ prime and } (p) = \frak{a}^2  \text{ for a fractional ideal } \fraka \mbox{ of }\FK \}  / \FK^{\times2}.   \]
\end{proposition}

\begin{proof} Let $z\FK^{\times2}$ represent a square class in $\Sel_2^\Q(F)$. Multiplying by an appropriate square, we can assume that $z \in \Z$ is squarefree and hence factors over the integers as $z = p_1\dots p_k$. Since $(z) = \mathfrak{a}^2$ for a fractional ideal $\fa$ of $\FK$, then for each $i$ we have $(p_i) = \mathfrak{a_i}^2$ for some fractional ideal $\fa_i$ of $\FK$. \end{proof}

\subsection{Classification for the image of $\Sel_2^\Q(F)$} \label{subsec::Imprimtypes}

We now classify the image of the representative for $\Sel_2^\Q(\FK)$ in Proposition \ref{prop:reps} under the map $\sgn : \Sel_2(\FK) \to V(\FK)$.  The bilinear spaces in Lemma \ref{lem:bilspace} contain different types of vectors, specifically:  \begin{itemize}
\item The alternating bilinear space in Lemma \ref{lem:bilspace} $(2)(a)$ contains 2 different types of vectors: the zero vector and nonzero vectors.
\item The nonalternating bilinear space in Lemma \ref{lem:bilspace} $(2)(b)$ contains 3 different types of vectors: the zero vector, the canonical vector, and noncanonical nonzero vectors.
\end{itemize}

\begin{proposition} \label{thm:Imageoffamilies} Let $p \in \Z$ be an odd prime. The image of the representatives given in Proposition \ref{prop:reps} under the map $ \sgn : \Sel_2(\FK) \to V_\infty(\FK) \boxplus V_2(\FK)$ are as follows: 
\begin{itemize}
\item If $V_2(\FK)$ is nonalternating then
\begin{eqnarray*} \sgn(-1) & = &  \: \: (\vect_{\can},\vect_{\can}); \\ \text{ If $2 \in \Sel_2^\Q(\FK)$ then } \sgn(2) & = & \begin{cases} (0,0) &  \sgn_2(2) = 0; \\  (0,\vect_{\can}) &  \sgn_2(-2) = 0;  \\ (0,\vect) & otherwise;  \end{cases}  \\ \text{ If $p \in \Sel_2^\Q(\FK)$ then }  \sgn(p) & = & \begin{cases} (0,0) &  \text{primes } p \equiv 1 \pmod 4; \\ (0,\vect_{\can}) &  \text{primes } p \equiv 3 \pmod 4.  \end{cases}  \end{eqnarray*} 
\item If $V_2(\FK)$ is alternating then
\begin{eqnarray*} \sgn(-1) & = &  \: \: (\vect_{\can}, 0); \\ \text{ If $2 \in \Sel_2^\Q(\FK)$ then }  \sgn(2) & = & \begin{cases} (0,0) & \sgn_2(2) = 0;  \\ (0,\vect)  & otherwise;  \end{cases}  \\ \text{ If $p \in \Sel_2^\Q(\FK)$ then } \sgn(p) & = & (0,0).  \end{eqnarray*} 
where $w_{\can}$ is the canonical vector and $w$ is a noncanonical nonzero vector. 
\end{itemize} \end{proposition}

\begin{proof} We consider representatives determined by primes $p \in \Z$ with $(p) = \frak{a}^2$. These classes have trivial image in $V_\infty(\FK)$ and so it remains to determine their image in $V_2(\FK)$. If $p$ is an odd prime, then either: $p \equiv 1 \pmod 4$ which means that for each place $v \mid 2$ then $p \in 1+4\cO_v$ and so $\sgn_2(p) = 0$; or  $p \equiv 3 \pmod 4$ which means that for each place $v \mid 2$ then $-p \in 1+4\cO_v$ and so $\sgn_2(p) = \vect_{\can}$. When $p =2$ there are three possibilities: $\sgn_2(2) = 0$, $\sgn_2(2) = \vect_{\can}$ which means that $\sgn_2(-2) = 0$, or $\sgn_2(2) = \vect$ for some noncanonical nonzero vector in $w \in V_2(\FK)$.  \end{proof}

\begin{definition} The \defi{$\Q$-imprimitive type} of a number field $\FK$ is the image $S^\Q(\FK) \coloneqq \sgn(\Sel_2^\Q(\FK))$
of the $\Q$-imprimitive part of the 2-Selmer group inside the 2-Selmer signature space. 
\end{definition}


\begin{theorem}[Classification] \label{thm::Q-types} There are 6 possibilities for $S^\Q(\FK)$, two types {\normalfont A(i)$-$(ii)} occur when $V_2(\FK)$ is alternating and four types {\normalfont B(i)$-$(iv)} occur when $V_2(\FK)$ is nonalternating. The following table provides a basis for each case. \normalfont 
\def\arraystretch{1.5} 
\begin{center}
\begin{tabular}{ c | c | l  }
\normalfont $\Q$-imprimitive type & $V_2(F)$ Alternating  & Generating set for $S^\Q(F)$ \\ \hline \hline
A(i) & Yes & $(\vect_{\can}, 0)$   \\  \hline
A(ii) & Yes & $(\vect_{\can}, 0), (0,\vect)$  \\  \hline
B(i) & No &  $(\vect_{\can}, \vect_{\can})$   \\  \hline
B(ii) & No & $(\vect_{\can}, \vect_{\can}), (0, \vect)$  \\  \hline
B(iii) & No & $(\vect_{\can}, \vect_{\can}), (0, \vect_{\can})$ \\  \hline
B(iv)& No & $(\vect_{\can}, \vect_{\can}), (0, \vect_{\can}), (0,\vect)$   \\  \hline
\end{tabular}
\end{center}
\textit{The following table lists the different conditions necessary for each $\Q$-imprimitive type. In the case that a possibility is excluded, we mark it with an} \xmark. \normalfont 
\def\arraystretch{1.5}
\begin{center}
\begin{tabular}{ c | c | c | c | c | c  } 
 \makecell{Prime $p \in \Sel_2^\Q(F)$ \\ with $p \equiv 3 \mod 4$} & $2  \in \Sel^\Q_2(F)$ & $\sgn_2(2) = 0$ & $\sgn_2(-2) = 0$  & $\sgn_2(-1) = 0$ &  $\Q$-imprimitive type \\  \hline \hline
\multirow{6}{*}{No}  & No & \xmark & \xmark & No  & B(i)  \\ \cline{2-6}
& \multirow{5}{*}{Yes}  &  No & No & No  &  B(ii)  \\  \cline{3-6}
&  & Yes & No & No  &  B(i)  \\  \cline{3-6}
& &  No & Yes & No & B(iii)  \\   \cline{3-6}
& &  No & No & Yes & A(ii) \\   \cline{3-6}
& &  Yes & Yes & Yes & A(i)  \\  \hline \hline
\multirow{6}{*}{Yes}  & No & \xmark & \xmark & No  & B(iii) \\ \cline{2-6}
& \multirow{5}{*}{Yes}  &  No & No & No  & B(iv) \\  \cline{3-6}
&  & Yes & No & No  &  B(iii)  \\  \cline{3-6}
& &  No & Yes & No & B(iii)  \\   \cline{3-6}
& &  No & No & Yes & A(ii) \\   \cline{3-6}
& &  Yes & Yes & Yes & A(i)  \\  \hline \hline
\end{tabular}
\end{center}
\end{theorem}

\begin{proof}
The proof comes from a case analysis using the table above and Proposition \ref{thm:Imageoffamilies}. 
 \end{proof}
 
 \begin{remark} When $r_1 = 0$ then $V_\infty(F) = \{0\}$ is trivial and so $\vect_{\can} = 0 \in V_\infty(F)$. In this case, then type B(i) and type B(iii) are the same and type B(ii) = type B(iv) are the same as well. 
 \end{remark}

\section{Conjectures: Isotropy ranks} \label{Section::5}

\begin{definition}
The \defi{isotropy rank} of a number field $\FK$ is the quantity
\[ k(\FK) \colonequals \dim_{\F_2} \Cl^+(\FK)[2]  - \dim_{\F_2} \Cl^+(\FK)[2] = \dim_{\F_2} S(\FK) \cap V_\infty(\FK)  \]
\end{definition}
From now on, let $\FK$  be an $S_n$-number field with $r_1$ real places and $2r_2$ complex places. Following the work in \cite[Section 7]{DV}, we make the following heuristic assumption:
\begin{enumerate}
\item[(H1)] As $\FK$ varies over $S_n$-number fields with fixed signature $(r_1,2r_2)$ and $\Q$-imprimitive type $S^\Q(\FK) \subseteq S(\FK)$ ordered by absolute discriminant, then $S(\FK) \subseteq V(\FK)$ is a random maximal totally isotropic subspace. \end{enumerate}  

By combining our heuristic assumption with the algebraic structure of the 2-Selmer group and the signature spaces, we develop a model for the distribution of isotropy ranks for each of the distinct $\Q$-imprimitive types in Theorem \ref{thm::Q-types}. Let
 \begin{eqnarray} \label{eqn::structure} V   \simeq   \I^2 \boxplus \cH^{r_1/2-1}  \: \:  & \text{ and } &  \: \: 
 W  \simeq  \begin{cases} \cH^{r_1/2 + r_2} & \text{ for types A(i)-(ii); } \\ \I^2 \boxplus \cH^{r_1/2 + r_2-1} & \text{ for  types B(i)-(iv). } \end{cases}   \end{eqnarray} 
where $r_1/2 \in \Z_{\geq 1}$ and $r_2 \in \Z_{\geq 0}$. Let $S \subset V \boxplus W$ be a maximal totally isotropic subspace with isotropy rank $k(S) \colonequals \dim_{\F_2} S \cap V$. Let $S^\Q \subset V \boxplus W$ be a fixed isotropic subspace corresponding to the $\Q$-imprimitive types whose basis is specified by the table in Theorem \ref{thm::Q-types}. 

\begin{theorem} \label{thm::isoprobs}
For $n,m,k \in \Z_{\geq0}$ and $\delta = \{0,1\}$, let 
\begin{eqnarray*} C_\delta (n,m,k) & \colonequals & \frac{1}{2^{k(k+m)}} \cdot  \frac{ (2)_{2n+m +\delta}  (4)_{n} (4)_{n+m} }{ (4)_{2n+m +\delta} (4)_{n-k}(2)_{k} (2)_{k + m}} \end{eqnarray*}
If $S \subset V \boxplus W$ is a maximal totally isotropic subspace with $S^\Q \subseteq S$ that is selected uniformly at random, then the probability that $S$ has isotropy rank $k$ is given by  
 
\begin{enumerate}[(I)]
\item Type {\normalfont A(i)} follows the distribution $C_0(r_1/2-1, r_2 +1,k-1) \text{ for } 1 \leq k \leq r_1/2.$ 
\item Type {\normalfont A(ii)} follows the distribution $C_0(r_1/2-1, r_2,k-1)$ for $1 \leq k \leq r_1/2$. 
\item Type {\normalfont B(i)} follows the distribution 
$$\begin{cases} C_1(r_1/2-1, r_2,k) & k = 0; \\
C_1(r_1/2-1, r_2,k) + C_1(r_1/2-1, r_2,k-1)/2^{r_1+r_2-1}  & 0<k < r_1/2; \\
C_1(r_1/2-1, r_2,k-1)/2^{r_1+r_2-1}  & k = r_1/2. \end{cases}$$. 
\item If $r_2 = 0$, then type {\normalfont B(ii)} follows the distribution 
$$\begin{cases} C_1(r_1/2-2, 1,k-1) & k = 1; \\
C_1(r_1/2-2, 1,k-1) + C_1(r_1/2-2, 1,k-2)/2^{r_1+r_2-2}  & 1<k < r_1/2; \\
C_1(r_1/2-2, 1,k-2)/2^{r_1+r_2-2} & k = r_1/2. \end{cases}$$
and otherwise it follows the distribution 
$$\begin{cases} C_1(r_1/2-1, r_2-1,k) & k = 0; \\
C_1(r_1/2-1, r_2-1,k) + C_1(r_1/2-1, r_2-1,k-1)/2^{r_1+r_2-2}  &0< k < r_1/2; \\
C_1(r_1/2-1, r_2-1,k-1)/2^{r_1+r_2-2} & k = r_1/2. \end{cases}$$ 
\item Type {\normalfont B(iii)} follows the distribution $C_0(r_1/2-1, r_2,k-1)$ for $1 \leq k \leq r_1/2$. 
\item If $r_2 = 0$, then type {\normalfont B(iv)} follows the distribution $C_0(r_1/2-2,1,k-2)$ for $2 \leq k \leq r_1/2$ and otherwise it follows the distribution $C_0(r_1/2-1, r_2-1,k-1)$ for $1 \leq k \leq r_1/2$.
\end{enumerate}
\end{theorem}

We postpone the proof of Theorem \ref{thm::isoprobs} until $\S$\ref{ss:max} where we develop the structure of maximal totally isotropic subspace inside an orthogonal direct sum. Combining Theorem \ref{thm::isoprobs} and our heuristic assumption (H1) yields the following conjecture.
 
\begin{conjecture} \label{conj::isotropyranks} 
As $\FK$ varies over $S_n$-number fields with signature $(r_1,r_2)$ and fixed $\Q$-imprimitive type chosen from A(i)-(ii) or B(i)-B(iv) ordered by absolute discriminant.
The density of fields with isotropy rank $k$ follows the distributions in Theorem \ref{thm::isoprobs}.
\end{conjecture}

\subsection{Maximal Totally Isotropic Subspaces in a Orthogonal Direct Sum} \label{ss:max}

 We now count maximal totally isotropic subspaces inside an orthogonal direct sum. A complete discussion can be found \cite[Appendix A]{DV}.

\begin{lemma}[Decomposition Lemma] \label{lem::Decomposition}
Let $V$ be a finite dimensional, nondegenerate, symmetric bilinear space over $\F_2$. For any vector $\vect \in V$, there is an orthogonal decomposition
\[ V = V_\vect \boxplus V_\vect^\perp \]
such that $\vect \in V_\vect$ and either: $V_\vect \simeq \I$ if $\vect$ is an anisotropic canonical vector, $V_\vect \simeq \I^2$ if $\vect$ is an isotropic canonical vector, or $V_\vect \simeq \cH$ otherwise. 
\end{lemma}

\begin{remark}  \label{rem::isosub} Let $V \simeq \I^2 \boxplus \cH^n$ or $V \simeq \cH^n$ for $n \in \Z_{\geq 0}$, let $S \subset V$ be an isotropic subspace, and let  $\vect \in S$ be a nonzero vector. If we form the orthogonal decomposition $V_\vect \boxplus V_\vect^\perp$ as in Theorem \ref{lem::Decomposition}, then $S$ can be decomposed into a direct sum of  two subspaces contained in the components of the orthogonal direct sum, explicitly $S = S_\vect \boxplus S_{\vect}^\perp$ where $ \langle \vect \rangle = S_\vect  \subset V_{\vect}$ and  $S_{\vect}^\perp \subset V_{\vect}^\perp$. \end{remark}

\begin{remark} \label{rem::concan}
Suppose that $V \simeq \I^2 \boxplus \cH^n$ for $n \in \Z_{\geq 1}$ and that $S \subset V$ is an isotropic subspace. By Lemma \ref{lem:bilspace} $(2)(b)$, the canonical vector $\vect_\can \in V$ is isotropic and further for any $v \in S$ then $b(v, \vect_\can) = b(v,v) = 0$, and hence $S' = \langle \vect_\can \rangle \oplus S$ is an isotropic subspace of $V$. Therefore, $\vect_\can$ is contained in every maximal totally isotropic subspace of $V$. 
\end{remark}

For $q,n \in \Z_{\geq 0}$, let \begin{equation*} (q)_n \colonequals \prod_{i=1}^n (1-q^{-i}),\hspace{1cm} (q)_0 \colonequals 1 \end{equation*} be the $q$-Pochhammer symbol.

\begin{lemma} \label{lem::maxtotiso}
For $t \in \Z_{\geq 0}$, the number of maximal totally isotropic subspaces inside both the bilinear space $\cH^t$ and the bilinear space $\I^2 \boxplus \cH^t$  is 
\begin{eqnarray*} b(t) & \colonequals & 2^{(t+1)t/2} \cdot \frac{(4)_{t}}{(2)_{t}}.  \end{eqnarray*}
\end{lemma}

\begin{proof} By \cite[Proposition A.6.1]{DV}, there are $2^{t^2} \prod_{i=1}^t(4^i -1)$ automorphisms of $\cH^t$ which act transitively on the set of maximal totally isotropic subspaces and $ 2^{t^2} \prod_{i=1}^t(2^i -1)$ of these automorphisms stabilize a specific maximal totally isotropic subspace. This concludes the count for $\cH^t$. In the case of $\I^2 \boxplus \cH^t$, observe that by remarks \ref{rem::isosub}-\ref{rem::concan} the maximal totally isotropic subspaces are of the form $S = \langle \vect_{\can} \rangle \oplus T$ where $T \subset \cH^t$ is a maximal totally isotropic subspace. \end{proof}

 We now refine Lemma \ref{lem::maxtotiso} to provide a count of the total number of maximal totally isotropic subspaces with a fixed isotropy rank.
\begin{lemma} \label{lem::maxtotisowithk} For $n,m,k \in \Z_{\geq 0}$, let
\begin{eqnarray*} d(n,m,k) & \colonequals & \frac{2^{(2n+m+1)(2n+m)/2}}{2^{k(k+m)}} \cdot  \frac{(4)_{n} (4)_{n+m} }{(4)_{n-k}(2)_{k} (2)_{k + m}}. \end{eqnarray*} 

 \begin{enumerate}
 \item The number of maximal totally isotropic subspaces $S \subseteq \left ( \cH^n \right) \boxplus \left ( \cH^{n+m} \right )$\\
  with $\dim_{\F_2} S \cap \left ( \cH^n \right) = k$ is given by $d(n,m,k)$ for $0 \leq k \leq n$.
 \item The number of maximal totally isotropic subspaces $S \subseteq \left ( \I^2 \boxplus  \cH^n \right ) \boxplus \left( \cH^{n+m} \right)$ \\
 with $\dim_{\F_2} S \cap  \left ( \I^2 \boxplus  \cH^n \right ) = k$ is given by $d(n,m,k-1)$ for $1 \leq k \leq n+1$.

\item The number of maximal totally isotropic subspaces $S \subseteq \left ( \I^2 \boxplus  \cH^n \right ) \boxplus   \left ( \I^2 \boxplus \cH^{n+m}  \right )$ \\
with $\dim_{\F_2} S \cap  \left ( \I^2 \boxplus  \cH^n \right )  = k$ which:

\begin{enumerate}[(a)] 
 \item contains $(\vect_\can,0)$ is given by $d(n,m,k-1)$ for $1 \leq k \leq n+1$, and
 \item does not contain $(\vect_\can,0)$ is given by $2^{2n+m+1}d(n,m,k)$ for $0 \leq k \leq n.$
\end{enumerate}
\end{enumerate}
\end{lemma}

\begin{proof} Let $S \subset V \boxplus W$ be a maximal totally isotropic subspace of isotropy rank $k$ and $S_V \colonequals S \cap V$ and $S_W \colonequals S \cap W$. Let $\mathcal{K}$ be a vector space complement for  $S_V$ in $(S_V)^\perp$. Let $\Aut(V)$ be the automorphism group and $\Aut(V,S_V)$ the subgroup that stabilizes $S_V$. By \cite[Equation A.19-A.20]{DV}, the number of these subspaces is 
\[ \frac{\#\Aut(V) }{  \#\Aut(V,S_V)} \#\Aut(\mathcal{K} ) \frac{ \#\Aut(W)}{\#\Aut(W,S_W)}. \]
If $V \simeq \cH^n$ and $W \simeq \cH^{n+m}$, then by \cite[Theorem A.13 and Corollary A.17(i)]{DV} we have $0 \leq k \leq n$ and $\dim_{\F_2} S_W = k + m$. Hence by \cite[Proposition A.6.1]{DV}, the above quantity  is 
\begin{eqnarray*} \frac{ \prod_{i=1}^n(4^i -1)}{ \prod_{i=1}^k(2^i -1) \prod_{i=1}^{n-k}(4^i -1)} 2^{(n-k)^2} \prod_{i=1}^{n-k}(4^i -1) \frac{ \prod_{i=1}^{m+n}(4^i -1)}{  \prod_{i=1}^{k+m}(2^i -1) \prod_{i=1}^{n-k}(4^i -1)} \end{eqnarray*}
which proves case (1). The other cases can be done by a similar analysis. \end{proof}



We can now combine Lemma \ref{lem::Decomposition} and Lemma \ref{lem::maxtotisowithk} to prove Theorem \ref{thm::isoprobs}.


\begin{proof}[Proof of Theorem \ref{thm::isoprobs}]  

Throughout, let $S \subset V \boxplus W$ be a maximal totally isotropic subspace of isotropy rank $k$ and use Lemma \ref{lem::maxtotiso}-\ref{lem::maxtotisowithk} to count the total number of possibilities for $S$ (with isotropy rank $k$).  

For type A(i), the total number of possibilities for $S$ is $b(r_1 + r_2 -1)$ and the number of possibilities with isotropy rank $k$ is determined by Lemma \ref{lem::maxtotisowithk}(2) with $n =r_1/2 -1$ and $m = r_2 +1$. 

For type A(ii), use Lemma \ref{lem::Decomposition} and remark \ref{rem::isosub} to write $W \simeq \cH^{r_1/2 + r_2-1 } \boxplus \cH$ such that every maximal totally isotropic subspace $S$ containing $(0,w)$ is of the form $S = T \oplus \langle (0,w) \rangle$ where $T \subset V \boxplus \cH^{r_1/2 + r_2-1 }$ is a maximal totally isotropic subspace of isotropy rank $k$. The number of possibilities for $T$ is $b(r_1 + r_2 -2)$ and the number with isotropy rank $k$ is determined by Lemma \ref{lem::maxtotisowithk}(2) with $n = r_1/2 -1$ and $m = r_2$. 

For type B(i), the total number of possibilities is $b(r_1 + r_2 -1)$ and the number of possibilities with isotropy rank $k$ is determined by Lemma \ref{lem::maxtotisowithk}(3) using $n = r_1/2 -1$ and $m = r_2$. For type B(ii), use Lemma \ref{lem::Decomposition} and remark \ref{rem::isosub} to instead count maximal totally isotropic subspaces $T \subset V \boxplus \left ( \I^2 \boxplus \cH^{r_1/2 + r_2-2} \right )$ with isotropy rank $k$. The total number of possibilities is $b(r_1 + r_2 -2)$ and we divide the number of possibilities with isotropy rank $k$ into two cases:
\begin{enumerate}
\item If $r_2 = 0$, swap the components such the $T \subset \left ( \I^2 \boxplus \cH^{r_1/2 - 2} \right ) \boxplus \left ( \I^2 \boxplus \cH^{r_1/2 - 1} \right )$ and then use Lemma \ref{lem::maxtotisowithk}(3) with isotropy rank $k'$ where $k' = k - 1$ by \cite[Theorem A.13]{DV}. 
\item If $r_2 > 0$,  use Lemma \ref{lem::maxtotisowithk}(3) with $n = r_1/2 - 1$ and $m = r_2-1$.
\end{enumerate}
For type B(iii),  the total number of possibilities is $b(r_1 + r_2 -2)$  and the number of possibilities with isotropy rank $k$ is determined by Lemma \ref{lem::maxtotisowithk}(3)(a) with $n = r_1/2 -1$ and $m = r_2$. For type B(iv), by Lemma \ref{lem::Decomposition} and remark \ref{rem::isosub} we instead count maximal totally isotropic subspaces $T \subset \cH^{r_1/2 - 1}  \boxplus \cH^{r_1/2 + r_2-2} $ of isotropy rank $k-1$. The total number of possibilities is $b(r_1 + r_2 -3)$ and the number with isotropy rank $k-1$ is divided into two cases:
\begin{enumerate}
\item If $r_2 = 0$, swap the components such that $T \subset  \cH^{r_1/2 - 2} \boxplus \cH^{r_1/2 - 1} $ and then use Lemma  \ref{lem::maxtotisowithk}(1) with isotropy rank $k'$ where $k' = k - 2$ by \cite[Theorem A.13]{DV}.
\item If $r_2 > 0$, use Lemma \ref{lem::maxtotisowithk}(1) with $n = r_1/2 -1$ and $m = r_2 -1$. 
\end{enumerate} 
This concludes the proof.\end{proof}

\section{Conjectures: $\Q$-imprimitive part of the 2-Selmer group} \label{Section::4}
We provide predictions for the asymptotic density of $S_n$-number fields with a given $\Q$-imprimitive type which stem from heuristics \cite{Bhar} local behaviors for $S_n$-number field. 

\subsection{Heuristics on local behaviors}

Let $\FK$ be a number field and define the local algebras 
\begin{itemize}
\item $\FK_\infty \colonequals \FK \otimes_\Q \R$, and 
\item $\AlF \colonequals \FK \otimes_\Q \Q_p$ for a prime $p \in \Z$. 
\end{itemize} 
For nonarchimedean local algebras, we write $\AlK \simeq \prod_{i=1}^r \FK_i$ where each $\FK_i$ is a finite extension of $\Q_p$.  The discriminant of a local algebra is $\Disc_\infty(\FK_\infty) \colonequals 1$ for archimedean local algebras and $\Disc_p(F_p) \colonequals  \prod_{i=1}^{r} \Disc_p(\FK_i)$ for nonarchimedean local algebras where $\Disc_p(\FK_i/\Q_p)$ denotes the highest power of $p$ that divides $\Disc(F_i/\Q_p)$.  

 As $F$ varies over $S_n$-number fields ordered by absolute discriminant, Bhargava  \cite{Bhar, Bhar1, Bhar2, Bhar4} provides conjectural results on the density of fields with any finite number of specifications (and \emph{acceptable} infinite sets of specifications \cite[Theorem 1.3]{Bhar4}) on the local algebras $F \otimes_\Q \Q_v$ for the places $v$ of $\Q$.  His conjectures are based on two assumptions:
  \begin{itemize} \item[(H2)] As $F$ varies over $S_n$-number fields ordered by absolute discriminant,  the density of fields with $F \otimes_\Q \Q_v \simeq F_v$ for a fixed local algebra $F_v$ is proportional to \[ \frac{1}{\Disc_v (F_v)} \cdot  \frac{1}{\# \Aut_{\Q_v} (F_v)}. \] 
   \item[(H3)] There is an independence of places i.e.  for two places $v \neq v'$ the fact that $F \otimes_\Q \Q_{v} \simeq F_v$ has no effect on the possibilities for $F \otimes_\Q \Q_{v'}$. \end{itemize} 
 
Let $c(n,p)$ be the weighted count for the total number of $\Q_p$-algebras of degree $n$ (up to isomorphism) where each algebra is weighted inversely proportional to $\Disc_p(\AlK)$ and $\# \Aut_{\Q_p}(\AlK)$.  Bhargava proves a mass formula \cite[Theorem 1.1]{Bhar} which establishes that \begin{equation} \label{eqn::density} c(n,p)  =  \sum_{k =0}^{n-1} q(k,n-k)p^{-k}  \end{equation}  where $q(k,n-k)$ denotes the number of partitions of $k$ into at most $n-k$ parts. Combining this result with the heuristic assumptions (H2) and (H3) yields his following main conjecture. 
\begin{conjecture}[Bhargava] \label{conj::mainconj} Let $\Sigma := (\Sigma_2,\Sigma_3, \Sigma_5, \dots)$ be an acceptable collection of local specifications where $\Sigma_p$ is a set of $\Q_p$-algebras. As $F$ varies over $S_n$-number fields with fixed signature $(r_1,r_2)$ ordered by absolute discriminant, then 
\begin{eqnarray*} \Prob(F  \mid F \otimes_\Q \Q_p \in \Sigma_p) = \prod_{p}  \frac{1}{c(n,p)} \left ( \sum_{\AlF \in \Sigma_p}  \frac{1}{\Disc(\AlF)} \cdot \frac{1}{\# \Aut_{\Q_p}(\AlF)} \right ) . \end{eqnarray*}
\end{conjecture}
Conjecture \ref{conj::mainconj}, and assumption (H2) and (H3) have been proved for $S_n$-number fields of degree $n \leq 5$ by \cite{Bhar1, Bhar2, Bhar3}. These heuristics for local different local behaviors will form the backbone for our predictions on $\Q$-imprimitive types. 

 \subsection{Conjectures: $\Q$-imprimitive types}
 
 We start with the density of $S_n$-number fields with a fixed prime $p$ in the 2-Selmer group. 
 
 \begin{conjecture} \label{conj::pinSel2} Let $p \in \Z$ be a prime.  As $F$ varies over $S_n$-number fields of degree $n = 2m$ with signature $(r_1,r_2)$ ordered by absolute discriminant,  then
 \[ \Prob( F \: | \: p \in \Sel_2^\Q(F) ) =  \frac{c(m,p)}{c(2m,p)}  \cdot \frac{1}{p^{m}}. \] 
  \end{conjecture}
  Let $S$ be finite set of primes and let
$T$ be a finite (\emph{or infinite} when $n > 2$) set  of primes  that is disjoint from $S$.  Let $X_{S,T}$ be the collection of $S_n$-number fields for which $p \in \Sel_2^\Q(F)$ for every prime $p \in S$ and $q \not \in \Sel_2^\Q(F)$ for every prime $q \in T$. 
  \begin{conjecture} \label{conj::psinSel2} As $F$ varies over $S_n$-number fields of degree $n = 2m$ with signature $(r_1,r_2)$ ordered by absolute discriminant, then 
\begin{eqnarray*} \Prob( F \: | \: F \in X_{S,T} ) & = &  \prod_{p \in S}  \left ( \frac{c(m,p)}{c(2m,p)}  \cdot \frac{1}{p^{m}} \right ) \prod_{q \in T}    \left (1 - \frac{c(m,q)}{c(2m,q)}  \cdot \frac{1}{q^{m}} \right ). \end{eqnarray*} 
 \end{conjecture}

When $n >2$,  Conjecture \ref{conj::psinSel2} can be used to predict the density of $S_n$-number fields whose $\Q$-imprimitive part of the Selmer group is generated by -1 and \emph{any finite set} of primes $\{p_1, \dots, p_s \}$.  This differs from the case of quadratic fields, where the $\Q$-imprimitive part of the Selmer group is generated by -1 and the primes dividing the fundamental discriminant.  We now consider the $S_n$-number fields with $2 \in \Sel_2^\Q(\FK)$ and refine our predictions to include the image of $-1$ and $2$ in the 2-Selmer signature space.

\begin{conjecture} \label{conj::unramifiedquadraticextensions} As $F$ varies over $S_n$-number fields of degree $n = 2m$ and signature $(r_1,r_2)$ with $2 \in \Sel_2(F)$ ordered by absolute discriminant, then 
\begin{eqnarray*}
\Prob( F \: | \: \sgn_2(-1) = 0) & = &  \frac{c(m,2)}{c(2m,2)} \cdot \frac{1}{2^{2m}}; \\
\Prob( F \: | \: \sgn_2(2) = 0) & = & \frac{c(m,2)}{c(2m,2)}  \cdot \frac{1}{2^{3m}}; \\
\Prob( F \: | \: \sgn_2(-2) = 0) & = & \frac{c(m,2)}{c(2m,2)} \cdot \frac{1}{2^{3m}}.  \end{eqnarray*} 
These possibilities are mutually exclusive when $4 \nmid n$. Otherwise, set $n = 4m'$  and then 
\begin{eqnarray*}
\Prob ( F \: | \: \sgn_2(-1) = 0 \text{ and } \sgn_2(2) = 0 ) & =  & 
\frac{c(m',2)}{c(4m',2)}  \cdot \frac{1}{2^{8m'}}.
 \end{eqnarray*} 
 \end{conjecture}


Combining Conjectures \ref{conj::pinSel2}- \ref{conj::unramifiedquadraticextensions} yields the following  main conjecture on $\Q$-imprimitive types. 

\begin{conjecture} \label{conj:fullclass}
Let
\begin{eqnarray*}
a & =  &\Prob(F \: | \: 2 \in \Sel_2^\Q(F) ) \\
b & = & \Prob(F \: | \:  p \not \in \Sel_2^\Q(F) \text{ for all primes  } p \equiv 3 \text{ mod } 4 ) \\
c & = & \Prob( F \: | \: F(\sqrt{-1})/F \text{ is unramified above }2)   \\
d & = & \Prob(F \: | \:  F(\sqrt{2})/F \text{ is unramified above }2) \\
e & = & \Prob(F \: | \: F(\sqrt{-1},\sqrt{2})/F \text{ are unramified above }2) 
\end{eqnarray*} 
according to the formulas in Conjecture $\ref{conj::pinSel2}$ and $\ref{conj::unramifiedquadraticextensions}$.  As $F$ varies over $S_n$-number fields of degree $n$ with signature $(r_1,r_2)$ ordered by absolute discriminant,  then:
\def\arraystretch{1.3}
\normalfont 
\begin{center}
\begin{tabular}{ l | l | l  }
 Type \:\:  & Basis for $S^\Q$ & Density among $S_n$-number fields\\ \hline  \hline  
A(i) & $(v_{\can}, 0)$ & $e$  \\  \hline
A(ii) &  $(v_{\can}, 0), (0,w)$ & $c-e$   \\  \hline
B(i) &  $(v_{\can}, v_{\can})$ & $(1-a+d-e)(b)$   \\  \hline
B(ii) &  $(v_{\can}, v_{\can}), (0, w)$ &  $(a-c-2d+2e)(b)$  \\  \hline
B(iii) &  $(v_{\can}, v_{\can}), (0, v_{\can})$ &  $(1-a+d-e)(1-b) + (d -e)$ \\  \hline
B(iv)&  $(v_{\can}, v_{\can}), (0, v_{\can}), (0,w)$ &  $(a-c-2d+2e)(1-b)$ \\ \hline
\end{tabular}
\end{center}
\end{conjecture}

\begin{proof} This follows from a case analysis using Theorem \ref{thm::Q-types}. \end{proof}

\begin{theorem} \label{thm::degreetwoandfour}
Conjectures $\ref{conj::pinSel2}$ - $\ref{conj:fullclass}$ are theorems for $S_n$-number fields of degree $n =2,4$.
 \end{theorem}

\subsection{Notation}
Throughout, let $p \in \Z$ be a prime. Let $\AlK = \prod_{i =1}^r F_i$ be a $\Q_p$-algebra where $e_i$ is the ramification index and $f_i$ is the inertia degree of the component field $F_i/\Q_p$. 
 
 \begin{definition} \label{def::split} The \defi{splitting symbol} associated to the algebra $\AlK$ is
$$(\AlK,p) \colonequals (f_1^{e_1}, \dots, f_r^{e_r}).$$ 

\end{definition}

We say that the splitting symbol $(\AlK,p)$ in definition \ref{def::split} has degree $n \colonequals \sum_i e_if_i$ and discriminant $\Disc_p((\AlK,p)) \colonequals p^{k}$ where $k \colonequals \sum_i (e_i-1)f_i $ is the \defi{discriminant exponent}. The total number of automorphisms of the splitting symbol is $\# \Aut (\AlK,p) \colonequals \prod_{i=1}^r f_i \cdot \frac{1}{N}$ where $N$ is the number of permutations of the 
components $f_i^{e_i}$ that preserve the splitting symbol.

\begin{definition} \label{def::partitions}
The \defi{partition} associated to a splitting symbol $(\AlK,p)$ is  
\[\lambda((\AlK,p)) \colonequals \underbrace{(e_1-1) + \dots + (e_1-1)}_{f_1 \text{ times }} + \dots + \underbrace{(e_r-1) + \dots + (e_r-1)}_{f_r \text{ times }} \]
which is a partition of the discriminant exponent $k$ into exactly $n-k$ parts.
\end{definition}

\begin{remark}
Definition \ref{def::partitions} differs from \cite{Bhar} in that we allow $0$ to be an entry in a partition and so $\lambda((\AlK,p))$ is a partition of the discriminant exponent $k$ into \emph{exactly} $n-k$ parts as opposed to \emph{at most} $n-k$ parts. See \cite{Ked} who makes the same distinction. 
\end{remark}

Let $\Spl(n)$ be the set of splitting symbols of degree $n$. For a discriminant exponent $0 \leq k \leq n-1$,  let $P(k,n-k)$ be the set of partitions of $k$ into at most $n-k$ parts.  For a fixed partition $\tau \in P(k,n-k)$,  let $\Sigma_\tau$ be the set of all $\Q_p$-algebras  $\AlK$ such that $\lambda((\AlK,p)) = \tau$.

\begin{theorem} \cite[Proposition 2.2]{Bhar} \label{thm::Bharparts} We have
 \begin{eqnarray*}\sum_{\AlK \in \Sigma_\tau } \frac{1}{\Disc_p(\AlF)} \cdot \frac{1}{\# \Aut_{\Q_p}(\AlF)} & =  & \frac{1}{p^k}. \end{eqnarray*} 
\end{theorem}

Let $\Spl(n)_\text{even}$ be the set of splitting symbols of degree $n$ with \emph{even} ramification indices, and let $P(k,n-k)_\text{odd}$ be the set of partitions of $k$ into exactly $n-k$ \emph{odd} parts.


\begin{lemma} \label{lem::bijection} For $m \in \Z_{\geq 1}$. The following map on splitting symbols is a bijection. 
\begin{eqnarray*} \Phi \colon \Spl(2m)_\text{even} & \to &  \Spl(m) \\
(f _1^{2e_1}, \dots ,  f_r^{2e_r}) & \mapsto  & (f _1^{e_1}, \dots ,  f_r^{e_r}). \end{eqnarray*} 
Let $0\leq k \leq m-1$ and set $k_0 = m +k$. The following map on partitions is a bijection. 
 \begin{eqnarray*} 
\Psi \colon \text{P}(k_0,2m-k_0)_\text{odd}  & \to &  \text{P}(k,m-k) \\
(2e_1 -1) + \hdots +  (2e_r -1) & \mapsto & (e_1 -1) + \hdots + (e_r -1).  \end{eqnarray*} \end{lemma} 

\subsection{Conjecture \ref{conj::pinSel2}}
Let $F$ be a number field of degree $n=2m$ with ring of integers $\Z_F$.  Let $p \in \Z$ be a prime and let $F_p = F \otimes \Q_p$.  Then the ideal $p\Z_F$ factors as the square of a fractional ideal if and only if $(F_p,p) \in S(2m)_\text{even}$.  In this case,  the sum of the residue degrees $\sum_i f_i$ is less than or equal to $m$ and therefore the discriminant exponent $k = 2m- \sum_i f_i$ will be greater than or equal to $m$.  Therefore, $\lambda((F_p,p)) \in P(k,2m-k)_\text{odd}$ for $m \leq k \leq 2m -1$.

Let $\Sigma_{even}$ be the set of local algebras $F_p$ with $(F_p,p) \in S(2m)_\text{even}$. By Theorem \ref{thm::Bharparts} and Lemma \ref{lem::bijection} (2)  then \begin{eqnarray*} \sum_{\AlF \in \Sigma_{\even}}  \frac{1}{\Disc_p(\AlF)} \cdot \frac{1}{\# \Aut_{\Q_p}(\AlF)}  & = &  \sum_{k_0 = m}^{2m-1} \# P(k_0,2m-k_0)_\text{odd} \cdot p^{-k_0} \\
& = & \sum_{k = 0}^{m-1} \# P(k,m-k) \cdot p^{-(k+m)} \\  
& = & \sum_{k = 0}^{m-1} q(k,m-k) \cdot p^{-k} \cdot p^{-m} \: = \: c(m,p)\cdot p^{-m} .  \end{eqnarray*}
By Equation \ref{eqn::density}, the total density of $\Q_p$-algebras of degree $2m$ is $c(2m,p)$ and dividing these quantities yields Conjecture \ref{conj::pinSel2}. 

\subsection{Conjecture \ref{conj::unramifiedquadraticextensions}}

Let $F$ be an $S_n$-number field of degree $n = 2m$. For a prime $p \in \Z$, let $\AlF = F \otimes_\Q \Q_p$ be the corresponding local algebra and note that $\AlF  \simeq \prod_{v|p} F_v$ where $F_v$ is the completion of $F$ at a place $v$. 

Let $K$ be an extension of $F$ with minimal polynomial $f(x) \in F[x]$. For a place $v \in F$, let $K_v = K \otimes_F F_v$ be the corresponding local algebra and note that $K_v  \simeq \prod_{w \mid v} K_w$ where $K_w$ is the completion of $K$ at a place $w$. Let $f(x) = f_1(x) \dots f_r(x)$ be a factorization of the minimal polynomial of $K$ into irreducibles over $F_v[x]$ and recall that the places $w \mid v$ correspond to the irreducible factors $f_i(x)$, explicitly $K_w \simeq K_i \colonequals  F_v[x]/f_i(x)$. The extension $K/F$ is unramified above $v$ if and only if the extension of local fields $K_w / F_v$ is unramified for each place $w \mid v$.

We now assume that $f(x)= x^2 -D$ where $D$ is the fundamental discriminant of a quadratic extension of $\Q$.  Ramification in the extension $K/F$ can only occur at places $v \mid D$,  and hence is completely determined by the component fields $F_v$ occurring in the local algebras $F_p \simeq \prod_{v \mid p} F_v$ for the primes $p \mid D$.  We now derive a formula to count the number of local algebras which are unramified when a square root of $d \in \Q_p^\times$ is adjoined. 

\subsubsection{Masses of local fields} We start with a statement of Serre's mass formula for local fields. 

\begin{theorem}[Serre's Mass formula] \label{prop::serrecountsram} Let $F$ be a finite extension of $\Q_p$ with a residue field of cardinality $q =p^{f_0}$. The mass of the totally ramified extensions $K/F$ of degree $e$ is \[ \sum_{\substack{ [K: F] = e  \\ \text{totally ramified}}} \frac{1}{\Norm_{F/\Q_p}  \Disc_q(K/F)} \cdot \frac{1}{\# \Aut_{F}(K)}  = \frac{1}{q^{e-1}}.  \]\end{theorem}

Here $\Norm_{F/\Q_p}  \Disc_q(K/F)$ is the highest power of $q$ dividing $\Norm_{F/\Q_p}  \Disc(K/F)$. The statement can be found for $F = \Q_p$ as \cite[Equation (1.2)]{Bhar} and see the Appendix for arbitrary local fields. We convert this into a \emph{weighted count} or \emph{mass} of local fields $K$ weighted by $1/\# \Aut_{\Q_p}(K)$ and $1/ \Disc_p(K)$ up to isomorphism over $\Q_p$.

\begin{proposition} \label{prop::countsram} Let $F$ be a finite Galois extension of $\Q_p$ with a residue field of cardinality $q =p^{f_0}$. The mass of the totally ramified extensions $K/F$ of degree $e$ is \[ \sum_{\substack{ [K: F] = e  \\ \text{totally ramified}}} \frac{1}{\Disc_p(K)} \cdot \frac{1}{\# \Aut_{\Q_p}(K)}  = \frac{1}{q^{e-1}} \cdot \frac{1}{[F:\Q_p]} \cdot \frac{1}{\Disc_p(F)^{e}}. \] \end{proposition}

\begin{proof} The mass of isomorphism classes over $F$ of totally ramified extensions $K/F$ is  \[ \sum_{\substack{ [K: F] = e  \\ \text{totally ramified}}}  \frac{1}{\Norm_{F/\Q_p} \Disc_q(K/F)} \cdot \frac{1}{\# \Aut_F(K)} = \frac{1}{q^{e-1}}. \] By \cite[Corollary (2.10)]{Neu} then $\text{Norm}_{F/\Q_p} \Disc_q(K/F) = \Disc_p(K) / \Disc_p(F)^{e}$. 
This yields \begin{equation*} \label{eqn::rand0} \sum_{\substack{ [K: F] = e  \\ \text{totally ramified}}}  \frac{1}{\Disc_p(K)} \cdot \frac{1}{\# \Aut_F(K)}  = \frac{1}{q^{e-1}} \cdot \frac{1}{\Disc_p(F)^{e}}. \end{equation*}
Let $W$ be a set of representatives for isomorphism classes of extensions $K$ over $F$ that will become isomorphic to the same extension $K$ over $\Q_p$. When $F$ is Galois, by \cite[Proposition (2.1)]{Bhar} then \[ \sum_{K \in W} \frac{1}{\# \Aut_{F}(K)} =  \frac{\# \Aut_{\Q_p}(F)}{\# \Aut_{\Q_p}(K)}  \]  
Arranging the earlier sum according to these sets $W$ for  each extension $K/\Q_p$  and the substituting the above result with $\#\Aut_{\Q_p}(F) = [F:\Q_p]$ concludes the proof. \end{proof}
 
 We now determine the mass of all local fields with prescribed ramification and inertia over a Galois extension $F/\Q_p$. 

\begin{proposition} \label{prop::counts} Let $F$ be a finite Galois extension of $\Q_p$ with residue field of cardinality $q =p^{f_0}$. The mass of the extensions $K/F$ with inertia degree $f$ and ramification index $e$ (counted up to isomorphism over $\Q_p$) is \[ \sum_{ \substack{[K: F] =  ef \\ \Inertia_{K/F} = f }} \frac{1}{\Disc_p(K)} \cdot \frac{1}{\# \Aut_{\Q_p}(K)} = \frac{1}{f} \cdot\frac{1}{q^{f(e-1)}} \cdot \frac{1}{[F :\Q_p]} \cdot \frac{1}{\Disc_p( F)^{ef}}. \]\end{proposition} \begin{proof}

Since there is a unique residue extension of each degree let $\tilde{L}$ be the unramified extension of degree $f \cdot f_0$ with residue field of cardinality $q^{f}$. Let $\tilde{F} = \tilde{L} F$ be the composition which is Galois since both $F, \tilde{L}$ are Galois. Instead of counting extensions of $F$ with inertia degree $f$ and ramification index $e$ it is equivalent to count totally ramified extensions of $\tilde{F}$ of degree $e$. By Proposition \ref{prop::countsram} then \begin{eqnarray*} \sum_{ \substack{ [K: F] = ef \\ \Inertia_{K/F} = f }} \frac{1}{\Disc_p(K)} \cdot \frac{1}{\# \Aut_{\Q_p}(K)} & = & \sum_{ \substack{ [K: \tilde{F}] = e \\ \text{totally ramified} }} \frac{1}{\Disc_p(K)} \cdot \frac{1}{\# \Aut_{\Q_p}(K)}  \\  & = &  \frac{1}{(q^f)^{(e-1)}} \cdot \frac{1}{[\tilde{F} :\Q_p]} \cdot \frac{1}{\Disc_p(\tilde{F})^{e} } \\  & = &  \frac{1}{q^{f(e-1)}} \cdot \frac{1}{f} \cdot \frac{1}{[F:\Q_p]} \cdot \frac{1}{\Disc_p(F)^{ef} }. \end{eqnarray*}   \end{proof}

Let $d \in \Q^\times_p$ be a fundamental discriminant. An immediate result of Proposition \ref{prop::counts} is the mass of the local fields $K/\Q_p$ with inertia degree $f$ and ramification index $2e$ such that $K[x]/(x^2-d) \simeq K \times K$. In fact, these fields have the same mass as the fields $K/\Q_p$ for which $K[x]/(x^2-d) \simeq K_{un}$ where $K_{un}$ is the unique unramified quadratic extension of $K$. 

\begin{lemma} \label{lem::unramextension} Let $d \in \Q^\times_p$ be a fundamental discriminant. The mass of the extensions $K/\Q_p$ with inertia degree $f$, ramification index $2e$, and either $K[x]/(x^2-d) \simeq K \times K$ or $K_{un}$ counted up to isomorphism over $\Q_p$ is  \begin{eqnarray*} \sum_{\substack{[K: \Q_p] = 2ef \\ \Inertia_{K/\Q_p} = f  \\ K[x]/(x^2-d) \simeq K \times K}} \frac{1}{\Disc_p(K)} \cdot \frac{1}{\# \Aut_{\Q_p}(K)} & = & \frac{1}{f} \cdot \frac{1}{p^{f(e-1)}} \cdot \frac{1}{ 2 \cdot \Disc_p( \Q_p(\sqrt{d}))^{ef} } \\\sum_{\substack{[K: \Q_p] = 2ef \\ \Inertia_{K/\Q_p} = f  \\ K[x]/(x^2-d) \simeq K_{un} }} \frac{1}{\Disc_p(K)} \cdot \frac{1}{\# \Aut_{\Q_p}(K)} & = &\frac{1}{f} \cdot \frac{1}{p^{f(e-1)}} \cdot \frac{1}{ 2 \cdot \Disc_p( \Q_p(\sqrt{d}))^{ef}}. \end{eqnarray*} \end{lemma}

\begin{proof} If $K$ is a local field such that $K[x]/(x^2-d) \simeq K \times K$, then $K$ contains $F = \Q_p(\sqrt{d})$ as a subfield. Applying Proposition \ref{prop::counts} to determine the density of fields $K/F$ of degree $[K : F] = ef$ with inertia degree $f$ and ramification index $e$ yields the desired result. 

Now let $K$ be a local field and suppose that $K(\sqrt{d}) \simeq K_{un}$. Let $\Un \subseteq K$ be the maximal unramified subfield over $\Q_p$ and let $\alpha \in \Un$ be an element such that $\Un(\sqrt{\alpha})/\Un$ is the unramified extension of degree 2 over $\Un$. Since $K(\sqrt{d})/K$ and $K(\sqrt{\alpha})/K$ are both unramified extensions of $K$ with degree 2 they are equal which means that $\sqrt{d\alpha} \in K$. Hence $K$ contains the ramified extension $F \colonequals  \Un( \sqrt{d\alpha})$. Note that $F/\Q_p$ is Galois since $F$ is contained in the abelian extension $\Un(\sqrt{d}, \sqrt{\alpha})/\Q_p$.
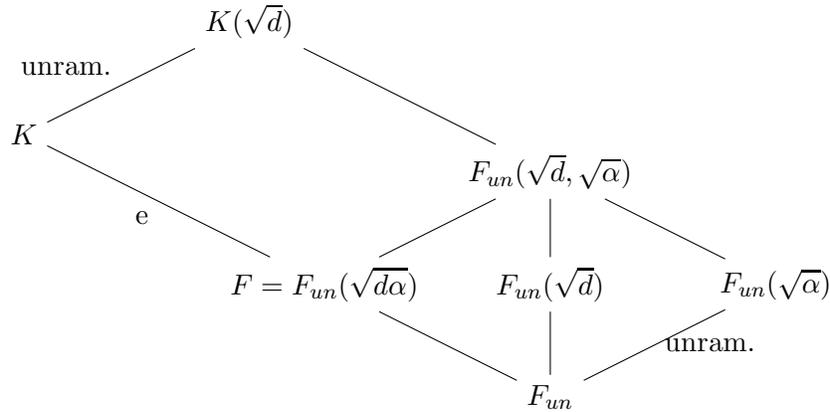
\[ \begin{tikzpicture}[node distance = 1.5cm, auto]
\node (Q) {$\Un$};
\node (A) [above of=Q, left of=Q, left of=Q ] {$F = \Un(\sqrt{d\alpha})$};
\node (B) [above of=Q, right of=Q, right of=Q] {$\Un(\sqrt{\alpha})$};
\node (C) [above of=Q] {$\Un(\sqrt{d})$};
\node (D) [above of=A, left of=A, left of=A,  node distance = 2cm] {$K$};
\node (K) [above of=Q, node distance = 3cm] {$\Un(\sqrt{d}, \sqrt{\alpha})$};
\node (L) [above of=K, left of=K,  left of=K, node distance = 2cm] {$K(\sqrt{d})$};
\draw[-] (Q) to node {} (A);
\draw[-] (Q) to node {} (C);
\draw[-] (C) to node {} (K);
\draw[-] (A) to node {e} (D);
\draw[-] (D) to node {\text{unram.}} (L);
\draw[-] (K) to node {} (L);
\draw (Q) -- ++ (B) node[midway,below,right]{\text{unram.}};
\draw[-] (A) to node {} (K);
\draw[-] (B) to node [swap] {} (K);
\end{tikzpicture} \]

Hence the local fields $K$ with inertia degree $f$ and ramification index $2e$ for which $K(\sqrt{d}) \simeq K_{un}$ are precisely the totally ramified extensions $K/F$ of degree $e$. Applying Proposition \ref{prop::countsram}, the density of totally ramified extensions $K/F$ is \begin{equation} \label{eqn::sumthin}  \sum_{\substack{[K \:: \:F] \: = \: e \\ K /F\text{ totally ramified }}} \frac{1}{\text{Disc}_p(K)} \cdot \frac{1}{\# \Aut_{\Q_p}(K)} = \frac{1}{p^{f(e-1)}} \cdot \frac{1}{2f} \cdot \frac{1}{\Disc_p(F)^{e}}. \end{equation} Since $\Un (\sqrt{\alpha},\sqrt{d})$ is an unramified extension of both $F$ and $\Q_p(\sqrt{d})$ then by  \cite[Corollary (2.10)]{Neu} \begin{eqnarray*}  \Disc_p(F)^2  & = &  \Disc_p(\Un(\sqrt{\alpha},\sqrt{d})) \\  & = &  \Disc_p(\Q_p(\sqrt{d}))^{2f}. \end{eqnarray*} Substituting $\Disc_p(F) = \Disc_p(\Q_p(\sqrt{d}))^{f}$ completes the proof. \end{proof}

\begin{remark} \label{rem::biquad1} Let $p = 2$. We will also be interested in the extensions $K/\Q_2$ with inertia degree $f$ and ramification index $4e$ such that both $K[x]/(x^2+1)$ and $K[x]/(x^2-2)$ are isomorphic to either of $K \times K$ or $K_{un}$.  Using the same notation as above, then $K$ must be an extension of 
\begin{eqnarray*} F & = & F_{un}(\sqrt{-1}, \sqrt{2}), \: \:  F_{un}(\sqrt{-\alpha}, \sqrt{2}), \: \:   F_{un}(\sqrt{-1}, \sqrt{2\alpha}),  \: \: \text{or} \;\: F_{un}(\sqrt{-\alpha}, \sqrt{2\alpha})  \end{eqnarray*} 
and by the same arguments the mass of these local fields is 
\begin{equation*} \sum_{K} \frac{1}{\text{Disc}_p(K)} \cdot \frac{1}{\# \Aut_{\Q_p}(K)} = \frac{1}{f} \cdot \frac{1}{p^{f(e-1)}} \cdot \frac{1}{4 \cdot \Disc_p(\Q_p(\sqrt{-1},\sqrt{2}))^{ef}}. \end{equation*}
\end{remark}

\subsubsection{Masses of local algebras}  Let $d \in \Q^\times_p$ be a fundamental discriminant. Let $\Sigma(2m,d)$ be the collection of local algebras $F_p$ up to isomorphism over $\Q_p$ with degree $n = 2m$ and such that each component field $F$ of the local algebra satisfies $F[x]/(x^2-d) \simeq F \times F$ or $F_{un}$. Observe that for each algebra $F_p \in \Sigma(2m,d)$ then $(\AlF,p) \in S(2m)_{even}$.

\begin{proposition} \label{prop::thecount}  Let $\sigma \in S(2m)_{even}$ and let $d \in \Q^\times_p$ be a fundamental discriminant.  The mass of the local algebras $F_p \in \Sigma(2m,d)$ with $(\AlF,p) = \sigma$ is
\[ \sum_{\substack{F_p \in \Sigma(2m,d) \\ (\AlF,p) = \sigma} } \frac{1}{\Disc_p(\AlF)} \cdot \frac{1}{\# \Aut_{\Q_p}(\AlF)} = \frac{1}{ \Disc_p( \Q_p(\sqrt{d}))^{m}} \cdot \frac{1}{\Disc_p(\Phi(\sigma))} \cdot \frac{1}{\#\Aut(\Phi(\sigma))}. \] \end{proposition}

\begin{proof} Same argument as \cite[Proposition 2.1]{Bhar}. If $F_p$ is a local algebra with $(\AlF,p) = \sigma = (f_1^{2e_1}, \dots, f_r^{2e_r})$ then the $i^{th}$-component field is a local field $F_i$ with inertia degree $f_i$ and ramification index $2e_i$. Counting such local algebras up to isomorphism is equivalent to counting ordered pairs of local fields $F_1 \times F_2 \times \dots \times  F_r$ weighted by $ \frac{1}{N} \prod_{i=1}^r \frac{1}{\# \Aut_{\Q_p}(F_i)}$ where $N$ is the number of permutation of $\sigma$ which preserve the $f_i^{2e_i}$.  

For each component, the number of possible local fields $F_i$ weighted inversely proportional to $\frac{1}{\# \Aut_{\Q_p}(F_i)} \cdot  \frac{1}{\Disc_p( F_i)} $ is given by Lemma \ref{lem::unramextension} as $$\sum_{\substack{[F_i: \Q_p] = 2e_if_i \\ \Inertia_{F_i/\Q_p} = f_i  \\ F_i[x]/(x^2-d) \text{ is unramified}}} \frac{1}{\Disc_p(F_i)} \cdot \frac{1}{\# \Aut_{\Q_p}(F_i)} = \frac{1}{f_i} \cdot \frac{1}{p^{f_i(e_i-1)}} \cdot \frac{1}{\Disc_p( \Q_p(\sqrt{d}))^{e_if_i} } $$ 
Hence the total mass of these local algebras is 
\[ \frac{1}{\Disc_p( \Q_p(\sqrt{d}))^{m}} \cdot \frac{1}{N}   \left (  \prod_{i =1}^r   \frac{1}{f_i} \cdot \frac{1}{p^{f_i(e_i-1)}}  \right )\]\end{proof} 

\begin{theorem} \label{thm::unrammifiedalgebras}The mass of the local algebras in $\Sigma(2m,d)$ is  \[ \sum_{ F_p \in \Sigma(2m,d) } \frac{1}{\Disc_p(\AlF)} \cdot \frac{1}{\# \Aut_{\Q_p}(\AlF)}  = \frac{1}{ \Disc_p( \Q_p(\sqrt{d}))^m}  \sum_{k =0}^{m-1} q(k,m-k) p^{-k } .\] \end{theorem}

\begin{proof} We start by rearranging the summation by clumping the algebras according to their splitting symbols and then applying Proposition \ref{prop::thecount} to get \begin{eqnarray*}  \sum_{ F_p \in \Sigma(2m,d)} \frac{1}{\Disc_p(\AlF)} \cdot \frac{1}{\# \Aut_{\Q_p}(\AlF)}  & = &  \sum_{\sigma \in S(m)} \left ( \sum_{ \substack{ F_p \in \Sigma(2m,d) \\ (\AlF,p) =  \Phi^{-1}(\sigma)}} \frac{1}{\Disc_p(F_p)} \cdot \frac{1}{\# \Aut_{\Q_p}(\AlF)} \right ) \\ & = & \frac{1}{ \Disc_p( \Q_p(\sqrt{d}))^{m}}  \left ( \sum_{ \substack{ \sigma \in S(m) }}  \frac{1}{\Disc_p(\sigma)} \cdot \frac{1}{\#\Aut(\sigma)} \right ). \end{eqnarray*} 
By \cite[Proposition 2.2 and Proposition 2.3]{Bhar} then 
\begin{eqnarray*} \sum_{ \substack{ \sigma \in S(m) }}  \frac{1}{\Disc_p(\sigma)} \cdot \frac{1}{\#\Aut(\sigma)} & = &  \sum_{k = 1}^{m-1} q(k,m-k) p^{-k }. \end{eqnarray*} \end{proof}

\begin{remark} \label{rem::biquad2}
Let $p = 2$ and let $\Sigma(4m',-1,2)$ be the collection of local algebras $F_2$ up to isomorphism over $\Q_2$ with degree $n = 4m'$ and such that each component field $F$ of the local algebra satisfies $F[x]/(x^2+1)$ and $F[x]/(x^2-2)$ are isomorphic to either of $F \times F$ or $F_{un}$. By Remark \ref{rem::biquad1} and a similar argument to Proposition \ref{prop::thecount} and Theorem \ref{thm::unrammifiedalgebras} we find that
 \[ \sum_{ F_2 \in \Sigma(4m',-1,2) } \frac{1}{\Disc_2(F_2)} \cdot \frac{1}{\# \Aut_{\Q_2}(F_2)}  = \frac{1}{ \Disc_2( \Q_2(\sqrt{-1}, \sqrt{2}))^{m'}}  \sum_{k =0}^{m'-1} q(k,m'-k) 2^{-k } .\]
\end{remark} 

Conjecture \ref{conj::unramifiedquadraticextensions} is now an immediate application of Conjecture \ref{conj::mainconj} using Theorem \ref{thm::unrammifiedalgebras} to determine the masses of the local algebras in $\Sigma(2m,-1)$, $\Sigma(2m,2)$, $\Sigma(2m,-2)$, and  $\Sigma(4m',-1,2)$.

\section{Conjectures: Average 2-torsion in class groups and narrow class groups}
We now make predictions on the average number of 2-torsion elements in class groups and narrow class group of $S_n$-number fields. We start with the following definition. 

\begin{definition}
Let $F$ be an $S_n$-number field. We say that $F$ has \defi{trivial $\Q$-imprimitive type} if $\Sel_2^\Q(K) = \{ \pm1 \}$ or equivalently no prime $p \in \Z$ factors as $(p) = \frak{a}^2$ in the field $F$.
\end{definition}

Based on computational data, we believe that Malle's heuristic on the 2-ranks of class groups extends to fields with trivial $\Q$-imprimitive type.  Hence we make the following conjecture. 

\begin{conjecture} \label{conj::classgroupmom}  \cite[Proposition 2.2]{Malle1}
As $F$ varies over $S_n$-number fields of signature $(r_1,2r_2)$ and trivial $\Q$-imprimitive type ordered by absolute discriminant, then
\begin{equation*} \Prob(F \: | \: \dim_{\F_2} \Cl(F) = \rho) = \frac{1}{ 2^{\rho(r_1+r_2-1) + \rho(\rho+1)/2} (2)_\rho } \frac{ (4)_{r_1+r_2 -1}  (2)_\infty }{ (2)_{r_1+r_2-1} (4)_\infty  }  \end{equation*}
with $n^{th}$ higher moments
\begin{equation*} \label{eq::moments} \prod_{i =1}^n (1+2^{i-r_1-r_2} ). \end{equation*}
\end{conjecture}

In particular, Conjecture \ref{conj::classgroupmom} implies that the average number of 2-torsion elements in $\Cl(F)$ is $1+2^{1-r_1-r_2}$.  By assuming that the isotropy rank of $F$ is independent from the size of the 2-torsion subgroup of the class group, we can use Theorem \ref{thm::isoprobs} to predict the average size of the 2-torsion subgroup for the narrow class group. 

\begin{conjecture} \label{cor::genusunobstructedtwotorsionelts} As $\FK$ varies over $S_n$-number fields with signature $(r_1,r_2)$ and trivial $\Q$-imprimitive type ordered by absolute discriminant then the average number of 2-torsion elements in $\Cl^+(\FK)$ is $ 1+ 2^{-r_2}$ when $r_1 > 0$ and $ 1+ 2^{1-r_2}$ when $r_1 = 0 $.  \end{conjecture}

\section{Examples/ Empirical results} \label{Section::6}

\subsection{Quadratic fields} 

We start with a classification of the $\Q$-imprimitive types of real quadratic fields along with there density. 

\begin{theorem} \label{thm::denquad} As $F$ varies over real quadratic fields ordered by absolute discriminant, then a density of  
$1/6$ have type $A(ii)$, $0\%$ have type $B(i)$, and $5/6$ have type $B(iii)$. 
\end{theorem}

\begin{remark} There are infinitely many real quadratic fields of type B(i), e.g., $\Q(\sqrt{p})$ for a prime $p \equiv 1 \bmod 4$. The number of such field with discriminant less than $X$ is asymptotic to $cX/\sqrt{\log(X)}$ where $c$ is an explicit constant (see \cite[p. 2036]{FouvryKluners} or \cite[p. 122]{Stevenhagen} which stems from \cite[Satz 3]{Rieger}). \end{remark}

\begin{proof} By \cite[p. 2036]{FouvryKluners}, then a proportion of $100\%$ of real quadratic fields $F$ have a prime $p \equiv 3 \pmod 4$ with $p \in \Sel_2^\Q(F)$. We now apply Conjecture \ref{conj::pinSel2} and \ref{conj::unramifiedquadraticextensions} (which are theorems since $n=2$) with the prime $p=2$ and explicitly write down the possibilities for the 2-adic algebras.
\def\arraystretch{1.5}
\begin{center}
\begin{tabular}{ l | l | l | c  }
 Prime 2 &  2-adic images &  \multicolumn{1}{c|}{ $F \otimes_\Q \Q_2$ }      &   Density   \\  \hline \hline
$2 \not \in \Sel_2(F)$ & NA & $\Q_2 \times \Q_2$ or $\Q_2(\sqrt{5})$ & 2/3  \\  \hline
 \multirow{ 3}{*}{$2 \in \Sel_2(F)$} & $\sgn_2(2) = 0$ &  $\Q_2(\sqrt{d})$  for $d = 2,10$ &  1/12  \\  \cline{2-4}
& $\sgn_2(-1) = 0$ &  $\Q_2(\sqrt{d})$  for $d = -1,-5$ &  1/6  \\  \cline{2-4}
& $\sgn_2(-2) = 0$ &  $\Q_2(\sqrt{d})$  for $d = -2,-10$ &  1/12   
\end{tabular}
\end{center}
Applying Theorem \ref{thm::Q-types} with $a  =  1/3$, $b  =  0$, $c = 1/6$, $d  =  1/12$, and $e  =  0$ establishes the density results. To conclude, let $F = \Q(\sqrt{D})$  be a quadratic field with fundamental discriminant $D>0$ such that for all primes $p \mid D$ then $p \equiv 1,2 \pmod 4$. In this case, $F \otimes_\Q \Q_2 \simeq \Q_2 \times \Q_2$ or $\Q_2(\sqrt{d})$ for $d = 2,5,10$ so these fields will all have $\Q$-imprimitive type B(i). 
\end{proof} 

This shows that $100\%$ of real quadratic fields have $k(F) = 1$. In fact, a much stronger statement can be made for real quadratic fields. 

\begin{lemma} \cite[Lemma 1]{FouvryKluners} Let $F = \Q(\sqrt{D})$ where $D >0$ is a fundamental discriminant. The following statements are equivalent: 
\begin{enumerate}
\item $\Cl^+(F) \simeq \Cl(F) \oplus \Z/2\Z$.
\item There is a prime $p \equiv 3 \pmod 4$ such that $p \mid D$. 
\end{enumerate}
\end{lemma}

\subsection{$S_4$-number fields.}

We now state the density of $\Q$-imprimitive types and our predictions for isotropy ranks for totally real $S_4$-number fields and mixed signature $S_4$-number fields (signature $(r_1,r_2) = (2,1)$) based on Theorem \ref{thm::degreetwoandfour} and Conjecture \ref{conj:fullclass}. 

\def\arraystretch{1.5}
\begin{center}
\begin{tabular}{ l || c | c | c | c  } 
\multicolumn{5}{c}{\textbf{Totally real $S_4$-fields}}  \\ 
Type     & Density     &  $k = 0$   & $k = 1$       & $k = 2$   \\ \hline \hline 
A(i)       & 0.0018      &  0             & 2/3               & 1/3         \\  \hline
A(ii)      & 0.0423      &   0             & 2/5              & 3/5         \\  \hline
B(i)       & 0.7280      & 16/45        & 26/45          & 3/45       \\  \hline
B(ii)      & 0.0996      &   0             & 2/5              &  3/5        \\  \hline 
B(iii)     & 0.1138       &  0              & 4/5              & 1/5        \\  \hline
B(iv)     & 0.0143      &   0              & 0                & 1          
\end{tabular}
\hspace{2cm}
\begin{tabular}{  l || c | c | c } 
\multicolumn{4}{c}{\textbf{Mixed signature $S_4$-fields}}   \\ 
Type     & Density    &  $k = 0$   & $k = 1$         \\ \hline \hline
A(i)       & 0.0018     &  0             & 1              \\  \hline
A(ii)      & 0.0423     &  0              & 1               \\  \hline
B(i)       & 0.7280     & 4/5            & 1/5            \\  \hline
B(ii)      & 0.0996     &  0              & 1               \\  \hline 
B(iii)     & 0.1138      &  0               & 1             \\  \hline
B(iv)     & 0.0143     &  0               & 1             
\end{tabular}
\end{center}

The density for each $\Q$-imprimitive type is determined by Conjecture \ref{conj:fullclass} using the values of  $a = 3/17$, $b \approx .87434$, $c = 3/68$, $d = 3/272$, and $e = 1/544$.

\subsection{$S_6$-number fields.} We now state our predications for the density of $\Q$-imprimitive types and isotropy ranks of totally real $S_6$-number fields based on Conjecture \ref{conj::isotropyranks} and Conjecture \ref{conj:fullclass}. 

\def\arraystretch{1.5}
\begin{center}
\begin{tabular}{ c || c | c | c | c | c }
\multicolumn{6}{c}{\textbf{Totally real $S_6$-fields}}  \\ 
Type & Density  & k = 0 & k = 1 & k = 2 & k = 3 \\ \hline \hline 
A(i) & $0.0000$ & 0 & 112/187 & 70/187 & 5/187 \\  \hline
A(ii)  & $0.0105$ & 0 & 16/51 & 30/51  & 5/51 \\  \hline
B(i) & $0.8848$ & 512/1683 & 976/1683 & 190/1683 &5/1683 \\  \hline
B(ii) & $0.0687$   & 0 & 32/51 &  18/51  &  1/51 \\  \hline
B(iii) & $0.0335$  & 0 &  16/51 & 30/51  & 5/51 \\  \hline
B(iv) & $0.0025$ & 0 & 0 & 2/3 & 1/3  \\  \hline
\end{tabular}
\end{center}

\subsection{$S_n$-number fields with type B(i) or trivial type} \label{subsec::trivialtypes}

Lastly, we conclude with a table for the numeric values of our predictions on the density of $S_n$-number fields with $\Q$-imprimitive type B(i) and the subset of these with trivial $\Q$-imprimitive type. 

\def\arraystretch{1.5}
\begin{center}
\begin{tabular}{ c || c | c | c | c | c | c | c | c | c  }
\multicolumn{9}{c}{\textbf{Density of $S_n$-number fields}}  \\ 
Type   & $n = 4$      & $n = 6$       & $n = 8$        & $n = 10$      & $n = 12$       &   $n = 14$      & $n = 16$      & $n = 18$    & $n = 20$      \\ \hline \hline 
B(i)     &  $0.7280$  &  $0.8848$   & $0.9434$       &  $0.9732$    &  $0.9863$     &   $0.9931$     &  $0.9965$   &  $0.9982$   & $0.9991$  	  \\  \hline
trivial  &  $0.6837$  &  $0.8762$   & $0.9417$     & $0.9728$      &  $0.9862$    &    $0.9931$     & $0.9965$     &  $0.9982$    & $0.9991$
\end{tabular}
\end{center}

\subsection{Computations: $S_4$-number fields} A \defi{quartic ring} is a commutative ring with $1$ that is free
of rank 4 as a $\Z$-module. Each $S_4$-number field contains a unique maximal quartic ring: the ring of integers. 
 In \cite{Bhar4},  Bhargava established a correspondence between quartic rings $R$ and pairs $A,B$ of integral ternary quadratic forms. The forms $A,B$ can be represented by two $3 \times 3$ symmetric matrices
$$ (A,B) :=  \frac{1}{2}  \left( \begin{pmatrix} 2a_{11} & a_{12}   & a_{13}   \\
 a_{12}  & 2a_{22}  & a_{23}  \\
 a_{13} & a_{23}  & 2a_{33} \end{pmatrix} ,   \begin{pmatrix} 2b_{11} & b_{12}   & b_{13}   \\
 b_{12}  & 2b_{22}  & b_{23}  \\
 b_{13} & b_{23}  & 2b_{33} \end{pmatrix}  \right) \hspace{1cm} a_{ij},b_{ij} \in \Z,  $$
up to equivalence by an action of $\text{GL}_2(\Z) \times \text{SL}_3(\Z)$. 

We developed an algorithm in Magma (available online at \cite{code}) to sample $S_4$-number fields with large discriminant. For a bound $X > 0$, we select a pair of matrices $(A,B)$ whose coefficients are chosen uniformly at random from the interval $[-X,X]$.  We  perform a series of tests \cite[pages 69-71] {Bhar3} on the coefficients of $(A,B)$ to confirm that the corresponding ring $R$ is a maximal order in an $S_4$-number field. We discard any rings that fail these tests. 

Let $F$ be the field of fractions for one of the remaining rings. The number of real places of $F$ corresponds to the points of intersection of the forms $A,B$ over $\mathbb{P}^2(\R)$.  If $F$ is totally real, there is fundamental domain for the action of $\text{GL}_2(\Z) \times \text{SL}_3(\Z)$ on $(A,B)$.  We can perform an additional test \cite[pages 78-79] {Bhar3} to ensure that our matrices $(A,B)$ are reduced with respect to this action. If $F$ is not totally real, then we perform a slower check to make sure that we haven't duplicated any fields. We repeat this process until we have collected a sufficiently large sample of fields.

\subsubsection{Isotropy ranks: totally real $S_4$-number fields} \label{subsec::S4}

We now test the predictions for the isotropy ranks of totally real $S_4$-number fields made in Conjecture \ref{conj::isotropyranks}. Using the method above, we collected a sample of $N = 10,000$ totally real $S_4$-number fields using the bound $X = 30$. The average size of the discriminants was approximately $9.33 \times 10^{19}$. 

\begin{center}
\def\arraystretch{1.5} 
\begin{tabular}{ l || c | c || c || c || c | c | c || c| c | c }
& \multicolumn{2}{c||}{Total $10,000$} & \multicolumn{2}{c||}{Density} & \multicolumn{3}{c||}{Predicted isotropy ranks} & \multicolumn{3}{c}{ {\color{blue} Actual isotropy ranks }} \\ \hline 
 Type  & $N_i$ & $1/\sqrt{N_i}$ & Predicted & {\color{blue} Actual} & $k = 0$ & $k = 1$ & $k = 2 $& {\color{blue} $k = 0$} & {\color{blue} $k = 1$ }& {\color{blue} $k = 2$ }\\ \hline \hline
A(i) & 26  & 0.196 & $0.0018$ & {\color{blue} $0.0026$ } &  0 & $0.667$ & $0.333$ & {\color{blue} 0 } & {\color{blue} 0.615 } & {\color{blue} 0.385 }  \\  \hline
A(ii) & 420  & 0.048 &  0.0423 & {\color{blue}  $0.0420$ } & 0 & $0.400$ & $0.600$ & {\color{blue}  0 } &  {\color{blue}  0.436} & {\color{blue} 0.564}  \\  \hline
B(i)  & 7168 & 0.012 & $0.7280$ & {\color{blue}  $0.7168$ } & $0.356$ & $0.578$ & $0.067$ &  {\color{blue}  0.361} & {\color{blue} 0.576 } & {\color{blue} 0.063 } \\  \hline
B(ii) & 1043 & 0.031 &  $0.0996$  & {\color{blue}  $0.1043$ }& 0 & $0.800 $ & $0.200$ & {\color{blue} 0 } & {\color{blue} 0.809} &  {\color{blue}  0.191} \\ \hline
B(iii) & 1195 & 0.029 &  $0.1138$  & {\color{blue}  $0.1195$  }& 0 & $0.400$ & $0.600$ & {\color{blue} 0 } & {\color{blue} 0.391} &{\color{blue} 0.609 } \\   \hline
B(iv) & 148 & 0.082 &  $0.0143$ & {\color{blue}  $0.0148$ } & 0 & 0 & 1 & {\color{blue} 0 } & {\color{blue} 0 } & {\color{blue} 1 } \\ \hline
\end{tabular}
\end{center}

\subsubsection{2-torsion: $S_4$-number fields with trivial $\Q$-imprimitive types.}  For each of the possible signatures $(4,0), (2,1)$ and $(0,2)$, we collected a sample of size $N =10,000$ of $S_4$-number fields with trivial $\Q$-imprimitive type using the bound $X = 40$. We then computed the 2-torsion of the class group, narrow class group, and the isotropy ranks to compare with Conjecture \ref{conj::classgroupmom} and \ref{cor::genusunobstructedtwotorsionelts}.

\def\arraystretch{1.5}
\begin{center}
\begin{tabular}{| c  || c | c || c | c || c | c || }
\hline
\multicolumn{1}{|c||}{} & \multicolumn{2}{c||}{\textbf{Totally real (4,0)}}  & \multicolumn{2}{c||}{\textbf{Mixed (2,1)}}  &   \multicolumn{2}{c||}{\textbf{Totally complex (0,2)}} \\ \hline 
& \textbf{Predictions} & \textbf{Actual}  & \textbf{Predictions} & \textbf{Actual}  & \textbf{Predictions} & \textbf{Actual}  \\ \hline
$\rho = 0$ & 0.8847 & 0.8891 & 0.7864 & 0.7888 & 0.6291 & 0.6333 \\ \hline
$\rho = 1$ & 0.1106 & 0.1073  & 0.1966 & 0.1937 & 0.3146 & 0.3099 \\ \hline
$\rho = 2$ &  0.0046 & 0.0036 & 0.0164 & 0.0165 & 0.0524 & 0.0527  \\ \hline
$\rho = 3$ & 0.00008  & 0.0000 & 0.0006 & 0.0009 & 0.0037 & 0.0040   \\  \hline
$k= 0$ & $\frac{16}{45} = 0.3556$ & 0.3593 & $\frac{4}{5}$ = 0.8000 & 0.8065  & 1 & 1 \\  \hline
$k= 1$ & $\frac{26}{45} = 0.5778$ & 0.5774 & $\frac{1}{5}$ = 0.2000 & 0.19344 & 0 & 0 \\   \hline
$k= 2$ & $\frac{3}{45}= 0.0667 $ & 0.0633 & 0 & 0 & 0  & 0 \\  \hline
$\Avg(\#\Cl(F)[2])$ & $1+\frac{1}{8}$  = 1.125 & 1.1181 & $1+\frac{1}{4}$ = 1.2500 & 1.2497 & $1+\frac{1}{2}$ = 1.5000 & 1.4975  \\  \hline
$\Avg(\#\Cl^+(F)[2])$ & $1+1$  = 2.000 & 1.9766 & $1+\frac{1}{2}$ = 1.5000 & 1.4899 & $1+\frac{1}{2}$ = 1.5000 & 1.4975    \\  \hline
\end{tabular}
\end{center}


\begin{thebibliography}{9}

\bibitem{ArmitageFrohlich}
J.V.\ Armitage and A.\ Fr\"ohlich, \emph{Class numbers and unit signatures}, Mathematika \textbf{14} (1967), 94--98.

\bibitem{BVV}
B.\ Breen, I.\ Varma, and J.\ Voight, \emph{On unit signatures and narrow class groups of odd abelian number fields: structure and heuristics}, appendix with Noam Elkies, arXiv:1910.00449.

\bibitem{BB} B.\ Breen, \emph{The 2-Selmer groups of number fields}, Ph.D. Thesis, 2020.


\bibitem{code}
B.\ Breen, \emph{2-Selmer group of number fields code}, 2020, available at \\
\href{https://github.com/BenKBreen/Sampling-S4-fields}{\texttt{https://github.com/BenKBreen/Sampling-S4-fields}}.

\bibitem{CohenLenstra}
H.\ Cohen and H.W.\ Lenstra, \emph{Heuristics on class groups of number fields}, Number theory, Noordwijkerhout (1983), Lecture Notes in Math, no.\ 1068, Springer, Berlin (1984), 33--62.

\bibitem{CohenMartinet}
H.\ Cohen and J.\ Martinet, \emph{\'Etude heuristique des groupes de classes des corps de nombres}, J.~Reine Angew.~Math.\ \textbf{404} (1990), 39--76.

\bibitem{DV}
D.S.\ Dummit and J.\ Voight, \emph{The $2$-Selmer group of a number field and heuristics for narrow class groups and signature ranks of units}, appendix with Richard Foote, Proc.\ London Math. Soc.\ \textbf{117} (2018), 682--726.
 
\bibitem{Bhar} 
M.\ Bhargava,  \emph{Mass Formulae for Extensions of Local Fields, and
Conjectures on the Density of Number Field Discriminants}
International Mathematics Research Notices, (2007) 

\bibitem{Bhar1}
M.\ Bhargava, \emph{The density of discriminants of quartic rings and fields}, Ann. of Math. (2), Vol. 162 No. 2, (2005), pp. 1031-1063.

\bibitem{Bhar2}
M.\ Bhargava, \emph{Higher Composition Laws III: The Parametrization of Quartic Rings}, Ann of Math.
(2), Vol. 159, No. 3 (2004), pp. 1329-1360

\bibitem{Bhar3}
M.\ Bhargava, \emph{Higher Composition Laws: Ph.D Thesis}, (June 2001), pp. 1-125.

\bibitem{Bhar4}
 M.\ Bhargava,\emph{The geometric sieve and the density of squarefree values of invariant polynomials}, arXiv preprint arXiv:1402.0031 (2014).

\bibitem{BV}
M.\ Bhargava, I.\ Varma, \emph{On the mean number of 2-torsion elements in the class groups, narrow
class groups, and ideal groups of cubic orders and fields}, Duke Math Journal, 164 (2015), 1911-1933.


\bibitem{HSV}W.\ Ho, A.\ Shankar, I.\ Varma,  \emph{Odd degree number fields with odd class number},
Duke Mathematical Journal, 167 (2018),  995-1047

\bibitem{EMP}
H.M.\ Edgar, R.A.\ Mollin, and B.L.\ Peterson, \emph{Class groups, totally positive units, and squares}, Proc.\ Amer.\ Math.\ Soc.\ \textbf{98} (1986), no.\ 1, 33--37.

\bibitem{Ha1}
R.\ Haggenm{\"u}ller, \emph{Signaturen von Einheiten und unverzweigte quadratische Erweiterungen totalreeler
Zahlk{\"o}rper}, Arch. Math. 39 (1982), 312-321.

\bibitem{Ha2}
R.\ Haggenm{\"u}ller, \emph{Diskriminanten und Picard-Invarianten freier quadratischer Erweiterungen},
Manuscripta Math. 36 (1981/82), no. 1, 83-103.

\bibitem{Janusz} 
G.\ Janusz,  \emph{Algebraic Number Theory}
Springer-Verlag Berlin Heidelberg New York, (1992)

\bibitem{JJ} 
J.\ Jones, D.\ Roberts, \emph{A database of local fields}
J. Symbolic Comput., 41 (1) (2006), pp. 80-97

\bibitem{FouvryKluners}
\'E.\ Fouvry and J.\ Kl\"uners, \emph{On the negative Pell equation}, Ann.\ of Math.\ \textbf{172} (2010), no.\ 3, 2035--2104.

\bibitem{Ked}
K. \ Kiran S. \emph{Mass formulas for local Galois representations (with an appendix by Daniel Gulotta)}, International Mathematics Research Notices 2007.9 (2007)

\bibitem{Lag1}
J.\ Lagarias, \textit{Signatures of units and congruences (mod 4) in certain real quadratic fields}. Reine Angew. Math.301 (1978), 142-146.

\bibitem{Lag2}
J.\ Lagarias,  \emph{Signatures and congruences (mod 4) in certain totally real fields}, Jour. f{\"u}r Math.
320 (1980), 1-5.

\bibitem{Lemmermeyer} 
F.\ Lemmermeyer, \emph{Selmer groups and quadratic reciprocity}, Abh.\ Math.\ Sem.\ Univ.\ Hamburg \textbf{76} (2006), 279--293.

\bibitem{LMFDB} 
The LMFDB Collaboration, The L-functions and Modular Forms Database, \url{http://www.lmfdb.org}, 2013, [Online; accessed 2016]

\bibitem{Magma}
W.~Bosma, J.~Cannon, and C.~Playoust, \emph{The Magma algebra system.\ I.\ The user language}, J.\ Symbolic Comput.\ \textbf{24} (3--4), 1997, 235--265.

\bibitem{Malle}
G.\ Malle, \emph{Cohen-Lenstra heuristic and roots of unity}, J.\ Number Theory \textbf{128} (2008), no.\ 10, 2823--2835. 

\bibitem{Malle1}
G.\ Malle, \emph{On the distribution of class groups of number fields}, Exp.\ Math.\ \textbf{19} (2010), vol.\ 4, 465--474.

\bibitem{Neu} 
J.\ Neukirch,  \emph{Algebraic Number Fields: Second Edition}
American Mathematical Society, (1996)

\bibitem{Rieger} 
G.\ J.\ Rieger. \textit{{\:U}ber die Anzahl der als Summe von zwei Quadraten darstellbaren und in einer primen Restklasse gelegenen Zahlen unterhalb einer positiven Schranke. II.} J.~Reine Angew.~ Math. (1965), 200--216.

\bibitem{PARI}
The PARI~Group, PARI/GP version {\tt 2.9.0}, Univ. Bordeaux, 2016,
\url{http://pari.math.u-bordeaux.fr/}.

\bibitem{Poon} 
Poonen, B. and Rains, E. \emph{Random maximal isotropic subspaces and Selmer groups}, J. Amer.
Math. Soc. 25 (2012), no. 1, 245-269.

\bibitem{Sage}
The Sage Developers, SageMath, the Sage Mathematics Software System (Version 7.1) (2016),
available at \url{http://www.sagemath.org/.}

\bibitem{Stevenhagen}
P.\ Stevenhagen, \emph{The number of real quadratic fields having units of negative norm}, Experiment.\ Math.\ \textbf{2} (1993), issue 2, 121--136.

\bibitem{JV}
J.\ Voight, \emph{Enumeration of totally real number fields of bounded root discriminant}, Algorithmic
number theory (ANTS VIII, Banff, 2008), eds. Alfred van der Poorten and Andreas Stein, Lecture

\end{thebibliography}
\end{document}